\documentclass[12pt]{amsart}
\usepackage{bbm}
\usepackage{amscd, amssymb}
\input{xy}
\xyoption{all}

\newtheorem{Thm}{Theorem}[subsection]
\newtheorem{Conj}[Thm]{Conjecture}
\newtheorem{Prop}[Thm]{Proposition}

\newtheorem{Def/Thm}[Thm]{Definition/Theorem}
\newtheorem{Cor}[Thm]{Corollary}
\newtheorem{Lemma}[Thm]{Lemma}

\theoremstyle{remark}
\newtheorem{Rmk}[Thm]{Remark}

\newtheorem{EG}[Thm]{Example}

\numberwithin{equation}{subsection}

\newcommand{\ot }{\otimes}
\newcommand{\ra }{\rightarrow}

\newcommand{\lra }{\longrightarrow}

\newcommand{\Hom }{{\mathrm{Hom}}}

\newcommand{\Spec}{{\mathrm{Spec}}}

\newcommand{\Pic}{{\mathrm{Pic}}}

\newcommand{\cO}{{\mathcal{O}}}

\newcommand{\cL}{{\mathcal{L}}}

\newcommand{\cF}{{\mathcal{F}}}

\newcommand{\cP}{{\mathcal{P}}}

\newcommand{\G}{{\bf G}}

\newcommand{\bT}{{\bf T}}
\newcommand{\bS}{{\bf S}}

\newcommand{\NN}{{\mathbb N}}

\newcommand{\PP }{{\mathbb P}}

\newcommand{\QQ }{{\mathbb Q}}
\newcommand{\CC }{{\mathbb C}}
\newcommand{\ZZ }{{\mathbb Z}}
\newcommand{\RR }{{\mathbb R}}

\newcommand{\one }{{\mathbbm 1}}

\newcommand{\ke }{{\varepsilon }}

\newcommand{\Mgk}{\overline{M}_{g,k}}

\newcommand{\fMgk}{\mathfrak{M}_{g,k}}

\newcommand{\QGraphe}{QG^\ke_{0,k,\beta}(\WmodG)}

\newcommand{\WmodG}{W/\!\!/\G}

\newcommand{\QmapX}{\mathrm{Q}_{g,k}(X,\beta)}
\newcommand{\QmapXe}{\mathrm{Q}_{g,k}^\ke(X,\beta)}

\newcommand{\QmapWe}{\mathrm{Q}_{g,k}^\ke(\WmodG,\beta)}
\newcommand{\Qgke}{\mathrm{Q}_{g,k}^\ke(\WmodG,\beta)}

\newcommand{\MgkW}{\overline{M}_{g,k}(\WmodG,\beta)}

\newcommand{\T}{{\bf {T}}}
\newcommand{\fM}{\mathfrak{M}}

\newcommand{\re}{\mathrm{e}}
\newcommand{\oM}{\overline{M}}

\newcommand{\lan}{\langle}
\newcommand{\ran}{\rangle}
\newcommand{\lla}{\langle\!\langle}
\newcommand{\rra}{\rangle\!\rangle}
\newcommand{\Atwo}{{\widetilde{A_2}}}

\begin{document}
\title[Higher genus quasimap wall-crossing]{Higher genus quasimap wall-crossing for semi-positive targets}

\begin{abstract} In previous work we have conjectured wall-crossing formulas for genus zero quasimap invariants of GIT quotients and proved them 
via localization in many cases. 
We extend these formulas to higher genus when the target is semi-positive, and prove them for semi-positive toric varieties, in particular for toric  local Calabi-Yau targets.  
The proof also applies to local Calabi-Yau's associated to some non-abelian quotients.
\end{abstract}

\author{Ionu\c t Ciocan-Fontanine}
\noindent\address{School of Mathematics, University of Minnesota, 206 Church St. SE,
Minneapolis MN, 55455, and\hfill
\newline \indent School of Mathematics, Korea Institute for Advanced Study,
85 Hoegiro, Dongdaemun-gu, Seoul, 130-722, Korea}
\email{ciocan@math.umn.edu}

\author{Bumsig Kim}
\address{School of Mathematics, Korea Institute for Advanced Study,
85 Hoegiro, Dongdaemun-gu, Seoul, 130-722, Korea}
\email{bumsig@kias.re.kr}
\maketitle

\section{Introduction} 
\subsection{Overview}
When a complex affine algebraic variety $W$ is acted upon by a reductive group $\G$, a choice of a 
character $\theta$ of $\G$ determines a linearization of the action,
and hence a GIT quotient $\WmodG=W/\!\!/_\theta \G$.

Under reasonable conditions on the triple $(W,\G,\theta)$, 
certain stability conditions, depending on a parameter $\ke\in \QQ_{>0}\cup \{0+,\infty\}$ produce (relatively) proper Deligne-Mumford moduli stacks of $\ke$-{\it stable quasimaps} from pointed
curves of genus $g$ to $\WmodG$,
carrying virtual fundamental classes. They come equipped with evaluation maps and with tautological cotangent $\psi$-classes at the markings and therefore determine (for projective $\WmodG$)
descendant $\ke$-quasimap invariants

\begin{equation}\label{invariants}
\langle \delta_1\psi_1^{a_1},\dots,\delta_k\psi_k^{a_k}
\rangle^\ke_{g,k,\beta}.
\end{equation}
As usual, if the target is only quasi-projective, but has a torus action with good properties, \eqref{invariants} are well-defined as equivariant invariants.
When $(g,k)\neq (0,0)$ and $\ke\in (1,\infty]$, \eqref{invariants} are the Gromov-Witten invariants of $\WmodG$. For $(g,k)=(0,0)$, the same holds when $\ke\in (2,\infty]$.

As $\ke$ varies, we expect the invariants to be related via wall-crossing formulas. For the genus zero sector, with arbitrary number of primary insertions and one descendant insertion,
such formulas are obtained in \cite{CK0}, where we also show that they may be interpreted as a vast generalization of Givental's toric mirror theorems. The genus zero wall-crossing
formulas are conjectured to hold for general $\WmodG$ and are proved in \cite{CK0} for many GIT targets by localization methods.
The present paper begins the exploration of wall-crossing at higher genus.
\subsection{Conjectures} Let
$${\bf t}={\bf t}(\psi)=t_0+t_1\psi+t_2\psi^2+t_3\psi^3+\dots,$$
where $t_j\in H^*(\WmodG,\QQ)$ are general even cohomology classes. Let the {Novikov ring} $\Lambda=\QQ[[q]]$
be the $q$-adic completion of the semigroup ring on the semigroup ${\mathrm {Eff}}(W,\G,\theta)$ of $\theta$-{\it effective} classes (see \S\ref{e-qmaps} for the definition).
The genus $g$ descendant potential of $\WmodG$ is defined by
$$F^\ke_g({\bf t}):=\sum_{\beta \in {\mathrm {Eff}}(W,\G,\theta)}\sum_{m\geq 0} \frac{q^\beta}{m!}\lan {\bf t}(\psi_1),{\bf t}(\psi_2),\dots, {\bf t}(\psi_m) \ran^\ke_{g,m,\beta}.$$
As usual, the unstable terms in the sum corresponding to $(g,m,\beta)$ for which the moduli spaces do not exist are omitted.

The wall-crossing formula relates the Gromov-Witten potential $F_g^\infty$ to the potential $F^\ke_g$ 
for {\it semi-positive} triples $(W ,\G, \theta)$ (for these triples, the
corresponding quotients $\WmodG$ have nef anti-canonical class). To state it, recall first from \cite{CK0} that in genus zero quasimap theory
there is a $J$-function $J^\ke(q,t,z)$ for each stability parameter $\ke\geq 0+$.
It depends on the Novikov variables, a general even cohomology element $t\in H^*(\WmodG,\QQ)$, and a formal variable $z$.
For $\ke >1$, it is the usual Givental (big) $J$-function of Gromov-Witten theory. The {\it small} $J^\ke$-function is defined as the restriction at $t=0$:
$$ J^\ke_{sm}(q,z):=J^\ke(q,0,z).$$
In the semi-positive case, $J^\ke$ takes values in $H^*(\WmodG,\Lambda)[[1/z]]$ and we will
need the first two terms in the $1/z$-expansion of small $J^\ke$,
$$ J^\ke_{sm}(q,z)=J^\ke_0(q)\one + J^\ke_1(q)\frac{1}{z}+O\left(\frac{1}{z^2}\right),$$
where $\one$ is the unit in cohomology. 
For the asymptotic stability $\ke=0+$ we use the special notation
 $I_{sm}=J^{0+}_{sm}$, and call this the small $I$-function of $\WmodG$.
The series $I_0(q)\in\Lambda$ is invertible, of the form $1+O(q)$, while the series $I_1(q)$ is in $H^{\leq 2}(\WmodG,\Lambda)$, with vanishing constant term in $q$.
For $\ke >0$, the coefficients $J_0^\ke$ and $J_1^\ke$ are polynomial $q$-truncations of $I_0$ and $I_1$. 

\begin{Conj}\label{main conj intro} If $\WmodG$ is a GIT quotient corresponding to a semi-positive triple $(W,\G,\theta)$,
then for every $\ke \geq 0+$ we have
\begin{equation}\label{GW to epsilon}
(J_0^\ke)^{2g-2}F_g^\ke({\bf t}(\psi))=  F^\infty_g\left(\frac{{\bf t}(\psi)+J_1^\ke}{J^\ke_0}\right).
\end{equation}
Further, for every $\ke_1\neq \ke_2$
\begin{equation}\label{all epsilon}
(J_0^{\ke_1})^{2g-2} F_g^{\ke_1}(J_0^{\ke_1}{\bf t}(\psi)-J_1^{\ke_1})= (J_0^{\ke_2})^{2g-2} F^{\ke_2}_g\left(J^{\ke_2}_0{\bf t}(\psi)-J_1^{\ke_2}\right).
\end{equation}
\end{Conj}
Since the transformation ${\bf t}(\psi)\mapsto ({\bf t}(\psi)+J_1^\ke)/J_0^\ke$ is invertible, the (apriori stronger) wall-crossing formula \eqref{all epsilon} follows from \eqref{GW to epsilon}.

Let $\{\gamma_1,\dots ,\gamma_s\}$ be a homogeneous basis of $H^{2*}(\WmodG,\QQ)$ and write $t_j=\sum_i t_{ji}\gamma_i$, so that $F^\ke_g$ are formal series in the $t_{ji}$ variables.
To have the uniform statement in the conjecture for all stability parameters, some modifications to the potentials are needed in low genus, as follows:
\begin{itemize}
\item
 When $g=0$ we must discard from both sides of equations \eqref{GW to epsilon} and \eqref{all epsilon} the parts of degree $\leq 1$ in the $t_{ji}$'s. There is, however, a natural 
way to include all terms of $t_{ji}$-degree $1$ in the genus zero potentials, so that \eqref{GW to epsilon} becomes an equality up to constants, see Remark \ref{degree one}. 
\item When $g=1$ a correction term is needed, to account for the failure of the dilaton equation on $\overline{M}_{1,1}$. Namely, we redefine 
\begin{equation}F^\ke_1({\bf t}):= \frac{1}{24}\chi_{\mathrm{top}}(\WmodG) \log J_0^\ke+\sum_{(\beta, m)\neq (0,0)} \frac{q^\beta}{m!}\lan {\bf t}(\psi_1),\dots, {\bf t}(\psi_m) \ran^\ke_{1,m,\beta},
\end{equation}
where $\chi_{\mathrm{top}}$ denotes the topological Euler characteristic.

\end{itemize}

 By matching Taylor coefficients in the $t_{ij}$'s
in Conjecture \ref{main conj intro} we get the equivalent formulation:
\begin{Conj}\label{equivalent} Let $\WmodG$ be a GIT quotient corresponding to a semi-positive $(W,\G,\theta)$. Fix $(g, n)\neq (1,0)$, with $2g-2+n\geq 0$, and a stability parameter $\ke\geq 0+$.
Then for arbitrary integers $a_1,\dots a_n\geq 0$ and arbitrary even cohomology classes 
$\delta_1,\dots \delta_n \in H^*(\WmodG,\QQ)$ 
\begin{align}\label{general}
&(J_0^\ke(q))^{2g-2+n}\sum_\beta q^\beta \langle \delta_1\psi_1^{a_1},\dots,\delta_n\psi_n^{a_n}
\rangle^\ke_{g,n,\beta}=\\
&\nonumber \sum_\beta q^\beta\sum_{m=0}^\infty \frac{1}{m!}\left\lan \delta_1\psi_1^{a_1},\dots,\delta_n\psi_n^{a_n}, \frac{J^\ke_1(q)}{J^\ke_0(q)},\dots, \frac{J^\ke_1(q)}{J^\ke_0(q)}
\right\ran_{g,n+m,\beta}^\infty.
\end{align}
If $(g,n)=(1,0)$ and $\ke$ is arbitrary,

\begin{equation}\label{unpointed genus 1}
\frac{1}{24}\chi_{\mathrm{top}}(\WmodG) \log J_0^\ke+\sum_{\beta\neq 0}\lan\;\ran^\ke_{1,0,\beta}=
\sum_{(\beta, m)\neq (0,0)} \frac{q^\beta}{m!}\left\lan \frac{J^\ke_1(q)}{J^\ke_0(q)},\dots, \frac{J^\ke_1(q)}{J^\ke_0(q)} \right\ran^\infty_{1,m,\beta}.
\end{equation}

\end{Conj}
Note that the string and divisor equations in Gromov-Witten theory allows one to rewrite 
the right-hand sides of \eqref{general}
and \eqref{unpointed genus 1} in terms of invariants with the same insertions as in the left-hand sides.

\subsection{Results} The first evidence we give for the conjectures is that they hold in genus zero for a large class of targets. 

\begin{Thm}\label{genus zero intro} Let $(W,\G,\theta)$ be semi-positive. Assume that $W$ admits an action by a torus $\T$, commuting with the action of $\G$ and such 
that the fixed points of the induced $\T$-action on $\WmodG$ are
isolated. Then the $g=0$ cases of Conjectures \ref{main conj intro} and \ref{equivalent} hold. 
Moreover, if $E$ is a convex $\G$-representation such that for all $\theta$-effective $\beta$ we have $\beta(\det(T_W))-\beta(W\times \det(E))\geq 0$, 
then the conjectures also hold at $g=0$ for the $E$-twisted $\ke$-quasimap
theories of $\WmodG$.

\end{Thm}
The inclusion of $E$-twisting means that Theorem \ref{genus zero intro} covers the compact Calabi-Yau targets which are realised as zero loci of sections of homogeneous vector
bundles on toric manifolds, on flag manifolds of classical types, or on products of such.
The proof of this result given in \S3 below can be summarized as follows: by first extending Dubrovin's genus zero reconstruction \cite{Dubrovin} to all $\ke$-theories, we reduce to
the case of big $J$-functions, which was already established in \cite{CK0}.

The second evidence is the main result of the paper.
\begin{Thm}\label{Main Thm}
Let $X$ be a nonsingular quasi-projective
toric variety of dimension $n$, obtained as the GIT quotient of a semi-positive triple $(\CC^{n+r}, (\CC^*)^r, \theta)$. Then Conjectures \ref{main conj intro} and \ref{equivalent} hold
for $X$.
\end{Thm}

The GIT presentation of a toric variety considered in Theorem \ref{Main Thm} is the standard one, coming from its fan, as in \cite{Cox}.  
(However, as explained in \S\ref{GIT presentations} later, the result holds for any other semi-positive
GIT presentation $(\CC^{n+r'}, (\CC^*)^{r'}, \theta')$ of $X$; note that the quasimap theories are different for different GIT presentations.)

It is easy to see that semi-positive toric varieties have $I_0(q)=1$.
When $X$ is projective and Fano, we have in addition (for the standard GIT presentation) that $I_1(q)=0$. The following Corollary is then an immediate consequence Theorem \ref{Main Thm}.
\begin{Cor} If $X$ is a nonsingular projective Fano toric variety, then its quasimap invariants are independent on $\ke$:
$$F_g^\ke({\bf t}(\psi))= F^\infty_g({\bf t}(\psi)),$$
for all $\ke\geq 0+$.
\end{Cor}

More interesting is the case of toric Calabi-Yau targets, for which Theorem \ref{Main Thm} is highly relevant to the Physicists' Mirror Symmetry (see e.g., \cite{HV}) at genus $g\geq 1$. 
For $X$ a
toric Calabi-Yau $3$-fold, consider the equality \eqref{GW to epsilon} for $\ke=0+$ and specialize at ${\bf t}(\psi)=0$. If we use the string and divisor equations,
the Gromov-Witten side
becomes precisely the $A$-model genus $g$ pre-potential {\it after} applying the mirror map. The Mirror Conjecture then implies that $F_g^{0+}|_{{\bf t}(\psi)=0}$ is {\it equal} to the $B$-model
genus $g$ pre-potential, expanded near a ``large complex structure" point for the mirror of $X$.

The proof of Theorem \ref{Main Thm} is given in sections 4 and 5 of the paper. In fact, we formulate and prove a stronger cycle-level comparison for the virtual classes under change of stability parameter from 
$\infty$ to some $\ke\geq 0+$. The first statement of this kind was established by Marian, Oprea, and Pandharipande in \cite{MOP} when the target $\WmodG$ is a Grassmannian.  
Our proof for toric targets is inspired by theirs, and in particular uses crucially a genus-reduction lemma from \cite{MOP}, but requires also several completely new ideas.

As we remark at the end of \S5, the proof we give works for other interesting non-compact Calabi-Yau GIT targets. In particular we obtain the following result.
\begin{Thm}\label{local Grass} Let $X$ be the total space of the canonical bundle over a Grassmannian, viewed as a GIT quotient in the canonical way.
Then Conjectures \ref{main conj intro} and \ref{equivalent} hold
for $X$.
More generally, the same is true for the total space of the canonical bundle over any type $A$ partial flag manifold.
\end{Thm}

\subsection{Aknowledgments}  
The research of I.C.-F. was partially supported by the NSA grant
H98230-11-1-0125 and the NSF grant DMS-1305004. The research of B.K. was partially supported by NRF-2007-0093859. In addition, I.C-F. thanks KIAS for financial support, excellent
working conditions, and an inspiring research environment. The authors thank the referee for helpful suggestions and Hyenho Lho for pointing out an error in an earlier version.

\section{Quasimap CohFT} To fix notation, in this section we recall briefly (after \cite{CKM}) the Cohomological Field Theory defined by the moduli spaces of $\ke$-stable quasimaps. 
\subsection{Quotients} \label{quotients}
Consider a triple $(W,\G, \theta)$ with $W$ an affine complex algebraic variety, $\G$ a reductive complex algebraic group acting on $W$, and $\theta\in\chi(\G)$ a character
of $\G$. The $\G$-equivariant line bundle $L_\theta=W\times \CC_\theta$, exhibits $\theta$ as a linearization of the action on the trivial line bundle. Hence $(W,\G, \theta)$ determines a GIT
quotient $\WmodG$. It has a projective morphism to the {\it affine quotient} $W/_{\mathrm {aff}}\G=\Spec(A(W)^\G)$. The line bundle $L_\theta$ descends to a relative polarization
on $\WmodG$, denoted by $\cO(\theta)$, which may be taken (without loss of generality) to be relatively very ample over $W/_{\mathrm {aff}}\G$.

We assume that 
the semi-stable and stable loci for the $\theta$-linearization coincide, $W^{ss}=W^s$, and that $\G$ acts freely on the stable locus. 
Further, we also assume that $W$ has at worst lci singularities and that
$W^s$ is smooth. Hence $\WmodG=[W^s/\G]$ is a nonsingular open substack in the quotient stack $[W/\G]$.

\subsection {$\ke$-stable quasimaps} \label{e-qmaps}
Let $C$ be a connected, at worst nodal, algebraic curve. A map $f:C\lra[W/\G]$ has a (homology) class 
$$\beta\in (\Pic ([W/\G]))^\vee=\Hom_\ZZ(\Pic ([W/\G]),\ZZ)=\Hom_\ZZ(\Pic ^\G(W),\ZZ)$$
given by 
$\beta(L)=\deg f^*L$. The map $f$ is a {\it quasimap to} $\WmodG$ if it sends the generic point of each irreducible component of $C$ to $\WmodG$. We may view the quasimap as a {\it rational} map from $C$ to $\WmodG$. The (finitely many) points of $C$ which are sent by $f$ to the complement of
$\WmodG$ in $[W/\G]$ are called the {\it base-points} of the quasimap.

A class $\beta\in (\Pic ([W/\G]))^\vee$ is
said to be $\theta$-{\it effective} if it is represented by a quasimap to $\WmodG$. For the purposes of this definition, the domain curve is allowed to have finitely many connected components.
The $\theta$-{effective} classes form a semigroup, denoted ${\mathrm {Eff}}(W,\G,\theta)$. By the boundedness
results in \cite{CKM}, for each $d>0$ the set
$$\{\beta\in {\mathrm {Eff}}(W,\G,\theta)\; |\: \beta(L_\theta)\leq d\}$$
is finite, see Remark 3.2.10 in loc.cit.

For each positive $\ke\in\QQ$ there is a stability condition on quasimaps;  
in addition, there is an asymptotic stability condition, denoted $\ke=0+$, in which $\ke$ is
allowed to be arbitrarily small (but still positive), see \cite{CKM}, also \cite{CK0}, \S2.4.

From now on we assume that $(g,k)\neq(0,0)$. For $\ke\geq 0+$, let $\QmapWe$ denote the moduli space of
$\ke$-stable quasimaps of class $\beta$ from $k$-pointed genus $g$ nodal curves to $\WmodG$. It is shown in \cite{CKM} that when the triple
$(W,\G,\theta)$ satisfies the assumptions of \S\ref{quotients} above, these moduli spaces are Deligne-Mumford stacks, proper over
the affine quotient $W/_{\mathrm {aff}}\G$, and carrying canonical perfect obstruction theories. The lci condition on $W$ is necessary for the perfectness of the obstruction theory
when $\ke\leq 1$. The space $\QmapWe$ is potentially nonempty
only when
$$2g-2+k+\ke \beta(L_\theta) >0.$$
The virtual dimension is
$$\mathrm{vdim}(\QmapWe)=\beta(\det (T_W)) + (1-g)(\dim(\WmodG)-3)+k,$$
where $T_W\in K^\circ_\G (W)$ is the (virtual) $\G$-equivariant tangent bundle. 

For each fixed class $\beta$ the set $\QQ_{>0}$ is divided into stability chambers by finitely many walls $1,\frac{1}{2},\dots,\frac{1}{\beta(L_\theta)}$ such that
the moduli spaces remain constant in each chamber.
When $\ke\in(1,\infty)$ one recovers the Kontsevich moduli spaces of stable maps to $\WmodG$; we write $\ke=\infty$ for these stability conditions. The asymptotic stability condition $\ke=0+$
corresponds to being in the first chamber $(0,\frac{1}{\beta(L_\theta)}]$ for {\it all} $\beta$.

For a long list of examples to which quasimap theory applies, see \S2.8 of \cite{CK0}. 

\subsection{$\ke$-quasimap invariants, $\ke$-CohFT} \label{projective}
Let $\QQ[ {\mathrm {Eff}}(W,\G,\theta)]$ be the semigroup ring. We write $q^\beta$ for the element corresponding to 
$\beta\in {\mathrm {Eff}}(W,\G,\theta)$.
The Novikov ring associated to the triple $(W,\G,\theta)$ is the $\mathfrak{m}$-adic completion
$$\Lambda:=\widehat{\QQ[ {\mathrm {Eff}}(W,\G,\theta)]}$$ 
with respect to the maximal ideal $\mathfrak{m}$ generated by $\{ q^\beta\; |\; \beta\neq 0\}$.
Throughout the paper $H^*(\WmodG,\QQ)$ and $H^*(\WmodG,\Lambda)$ will denote the {\it even} cohomology with the indicated coefficients. 

In this subsection we assume that $\WmodG$ is projective; the extension to quasi-projective targets is discussed in the next subsection. We have the intersection pairing 
$$\lan\gamma,\delta\ran:=\int_{\WmodG}\gamma\delta $$
on cohomology. We extend it $\Lambda$-linearly to $H^*(\WmodG,\Lambda)$.

Once and for all, fix homogeneous bases $\{\gamma_1,\dots,\gamma_s\}$ and $\{\gamma^1,\dots ,\gamma^s\}$ of $H^*(\WmodG,\QQ)$, dual with respect to the pairing $\lan \; , \;\ran$.
The cohomology class of the diagonal $\Delta\subset \WmodG\times\WmodG$ is then $\sum_{i=1}^s\gamma_i\otimes\gamma^i$.

\subsubsection{Brackets and double brackets}
Let $\psi_i$ be the first Chern class of the canonical line bundle on $\Qgke$ with fiber the cotangent line to the domain curve at the $i^\mathrm{th}$ marking.
For any stability parameter $0+\leq \ke\leq\infty$, the descendant invariants of $\ke$-quasimap theory are defined by
\begin{equation}\label{e-invariants}
\langle \delta_1\psi_1^{a_1},\dots,\delta_k\psi_k^{a_k}
\rangle^\ke_{g,k,\beta}=\int_{[\Qgke]^{\mathrm{vir}}} \prod_{i=1}^k ev_i^*(\delta_i)\psi_i^{a_i},
\end{equation}
where $\delta_i\in H^*(\WmodG,\QQ)$, $ev_i:\Qgke\lra\WmodG$ are the evaluation maps at the markings, and $g,k,a_1,\dots a_k\geq 0$ are integers such that  $2g-2+k+\ke\beta(L_\theta)> 0$.

\begin{Rmk} In Gromov-Witten theory, the name ``invariants" for the brackets \eqref{e-invariants} 
reflects the fact that they are symplectic invariants of the target $\WmodG$ and in particular do not change under 
deformations of the target $\WmodG$.
The quasimap brackets depend on the {\it pair} of stacks $([W/\G],[W^s/\G])$, see Proposition 4.6.1 in \cite{CKM}. Nevertheless, we will use the same terminology in quasimap theory as well.
\end{Rmk}

When dealing with generating series of invariants, it is convenient to use a double bracket notation
\begin{equation}\label{double bracket}
\begin{split}
&\lla \delta_1\psi_1^{a_1},\dots,\delta_k\psi_k^{a_k} \rra^\ke_{g,k}=\lla \delta_1\psi_1^{a_1},\dots,\delta_k\psi_k^{a_k} \rra^\ke_{g,k}({\bf t}(\psi)):=\\
&\sum_{m,\beta} \frac{q^\beta}{m!}
\langle \delta_1\psi_1^{a_1},\dots,\delta_k\psi_k^{a_k},
{\bf t}(\psi_{k+1}), \dots ,{\bf t}(\psi_{k+m}) \rangle^\ke_{g,k+m,\beta}.
\end{split}
\end{equation}
Here, as in the Introduction, ${\bf t}(\psi)=t_0+t_1\psi+t_2\psi^2+t_3\psi^3+\dots$, with $t_j=\sum_i t_{ji}\gamma_i\in H^*(\WmodG,\QQ)$.
The double brackets are formal functions on the {\it large phase space} with coordinates $t_{ji}$. We will often need to specialize them to $t_1=t_2=\dots =0$ and will write
$$\lla \delta_1\psi_1^{a_1},\dots,\delta_k\psi_k^{a_k} \rra^\ke_{g,k}(t)=\lla \delta_1\psi_1^{a_1},\dots,\delta_k\psi_k^{a_k} \rra^\ke_{g,k}\mid_{{\bf t}=t}$$ 
for the specializations, with $t:=t_0=\sum_i t_{0i}\gamma_i$.

Apriori the unstable terms are omitted from the sum \eqref{double bracket}, though in many cases certain conventions will be made to include them as well. When we do that,
the appropriate conventions will be spelled out explicitly. 
In this notation, the genus $g$ potential (when $g\geq 1$) is the empty bracket
$$F^\ke_g ({\bf t}(\psi))=\lla \;\; \rra_{g,0}^\ke ({\bf t}(\psi)),$$
and then \eqref{double bracket} with $\delta_j=\gamma_{i_j}$ is the derivative $\frac{\partial}{\partial t_{a_1i_1}}\dots \frac{\partial}{\partial t_{a_ki_k}}F^\ke_g ({\bf t}(\psi))$.

\subsubsection{$\ke$-quasimap classes}
Let $$f:\Qgke\lra\Mgk$$ be the composition of the forgetful morphism $\Qgke\lra\fMgk$ with the stabilization morphism $\fMgk\lra\Mgk$.
If $2g-2+k>0$, define $\Lambda$-linear maps
\begin{equation}\label{Omega}
\begin{split}
&\Omega^\ke_{g,k} :H^*(\WmodG,\Lambda)^{\otimes k}\lra H^*(\Mgk,\Lambda), \\& \Omega^\ke_{g,k}(\otimes_{j=1}^k\delta_j)=
\sum_\beta q^\beta \Omega^\ke_{g,k,\beta}(\otimes_{j=1}^k\delta_j),
\end{split}
\end{equation}
by setting
\begin{equation}\label{Omega b}
 \Omega^\ke_{g,k,\beta}(\otimes_{j=1}^k\delta_j)=f_*( [\Qgke]^{\mathrm{vir}}\cap\prod_{j=1}^k ev_j^*(\delta_j)).
\end{equation}
The maps \eqref{Omega} are clearly equivariant for the actions of the symmetric group $S_k$ on the source and target. 

\subsubsection{$\ke$-CohFT}
The boundary of $\Qgke$ is the complement of the open stratum of quasimaps with irreducible and nonsingular domain curve. It has a recursive structure, with strata indexed by modular graphs 
and the virtual classes behave in a functorial way with respect to this structure. These facts are well-known for the moduli spaces of stable maps,
and one sees immediately that their standard proofs, as given in \cite{Behrend}, are $\ke$-independent. The general statement we will
use is the one appearing as Theorem 13 in \cite{Getzler}, with the spaces $\MgkW$ replaced by general $\Qgke$. 

In particular, we obtain virtual divisors covering the boundary as follows. Let $[k]:=\{1,2,\dots,k\}$. Let 
$$\mathcal{B}(g,k):=\{((g_1,S_1),(g_2,S_2))\;|\; g=g_1+g_2,\;  [k]=S_1\coprod S_2\;  \}.$$
be the set of ordered partitions of $(g,k)$. For each $\sigma\in \mathcal{B}(g,k)$ put
$$\widetilde{D}^\ke_\sigma:=\coprod_{\beta=\beta_1+\beta_2} \mathrm{Q}^\ke_{g_1,S_1\cup\bullet}(\WmodG,\beta_1)\times_{\WmodG} \mathrm{Q}^\ke_{g_2,S_2\cup\star}(\WmodG,\beta_2),$$
where the fiber product is over the evaluation maps at the additional markings. Alternatively, $\widetilde{D}^\ke_\sigma$ is defined by the fiber product diagram
\begin{equation*}
\begin{CD}
{\widetilde {D}}^\ke_\sigma  @>>> \coprod_{\beta_1+\beta_2=\beta}\mathrm{Q}^\ke_{g_1,S_1\cup\bullet}(\WmodG,\beta_1)\times \mathrm{Q}^\ke_{g_2,S_2\cup\star}(\WmodG,\beta_2)\\
@VVV @VV ev_\bullet\times ev_\star V\\
\WmodG @ >\Delta >> \WmodG\times\WmodG .
\end{CD}
\end{equation*}
There is a proper gluing map
$$\begin{CD}
\widetilde{D}^\ke_\sigma@>h_\sigma>> \Qgke,
\end{CD}
$$
and the boundary divisor $D^\ke_\sigma$ is
the stack-theoretic image of $h_\sigma$.
The virtual classes of $\widetilde{D}^\ke_\sigma$ and of $D^\ke_\sigma$ are defined by
\begin{equation}\label{bdry1}
[\widetilde{D}^\ke_\sigma]^{\mathrm{vir}}:=\Delta^! \sum_{\beta_1+\beta_2=\beta}
[\mathrm{Q}^\ke_{g_1,S_1\cup\bullet}(\WmodG,\beta_1)]^{\mathrm{vir}}\otimes [\mathrm{Q}^\ke_{g_2,S_2\cup\star}(\WmodG,\beta_2)]^{\mathrm{vir}},
\end{equation}
\begin{equation}\label{bdry2}
[D^\ke_\sigma]^{\mathrm{vir}}:=(h_\sigma)_*[\widetilde{D}^\ke_\sigma]^{\mathrm{vir}}.
\end{equation}

Similarly, there is a proper gluing map 
$$ h_0:\WmodG\times_{\WmodG\times\WmodG}\mathrm{Q}^\ke_{g-1,[k]+\bullet+\star}(\WmodG,\beta)\lra\Qgke ,$$
(fiber product over the diagonal map $\Delta$ and the pair $(ev_\bullet, ev_\star)$) whose image gives a boundary divisor $D_0^\ke$ with virtual class 
\begin{equation}\label{bdry3}
[D^\ke_0]^{\mathrm{vir}}=(h_0)_* \Delta^! [\mathrm{Q}^\ke_{g-1,[k]+\bullet+\star}(\WmodG,\beta) ]^{\mathrm{vir}}.
\end{equation}

This leads to the usual splitting properties for the brackets \eqref{e-invariants} and the classes \eqref{Omega b}. For example, if 
$[k]=S_1\coprod S_2$, $g_1+g_2=g$, with $|S_i|=k_i$, is a stable ordered partition, then
\begin{equation}\label{splitting 1} \rho^*\Omega_{g,k}^\ke (\otimes_{j=1}^k\delta_j)=\sum_i \Omega_{g_1,k_1+1}^\ke ((\otimes_{j\in S_1}\delta_j)\otimes\gamma_i)
\Omega_{g_2,k_2+1}^\ke ((\otimes_{j\in S_2}\delta_j)\otimes\gamma^i),
\end{equation}
where
$$\rho:\overline{M}_{g_1,S_1\cup\bullet}\times \overline{M}_{g_2,S_2\cup\star}\lra\Mgk$$ is the gluing map on stable curves.
Analogously, 
\begin{equation}\label{splitting 2}
\phi^* \Omega_{g,k}^\ke (\otimes_{j=1}^k\delta_j)=\sum_i \Omega_{g-1,k+2}^\ke ((\otimes_{j=1}^k\delta_j)\otimes\gamma_i\otimes\gamma^i),
\end{equation}
where $(g-1,k+2)$ is stable and
$$\phi :\overline{M}_{g-1,k+2}\lra\Mgk$$
is again the gluing map.
In other words, we have the following 
\begin{Prop}\label{CohFT} For each $\ke\geq 0+$
the maps \eqref{Omega} give the structure of a Cohomological Field Theory (CohFT) over $\Lambda$ on $H^*(\WmodG,\Lambda)$
with the metric $\lan \; , \;\ran$. 
\end{Prop}

The primary invariants \eqref{e-invariants} in the stable range $2g-2+k>0$ are given by the degree zero part of the CohFT,
$$\sum_\beta q^\beta\langle \delta_1,\dots,\delta_k
\rangle^\ke_{g,k,\beta}=\int_{\Mgk} \Omega_{g,k}^\ke (\otimes_{j=1}^k\delta_j) .$$ 
The genus zero potential $F^\ke_0(t)$ satisfies the WDVV equation and
determines the $\ke$-quasimap Frobenius manifold. The $\ke$-quantum product is given by
$$ \gamma_i\circ_\ke \gamma_j=\sum_{l=1}^s \gamma_l\lla\gamma_i,\gamma_j,\gamma^l\rra^\ke_{0,3}(t). $$
The unit for this product is discussed in Remarks \ref{string etc} and \ref{unit} below. 

Including $\psi$-classes on $\Mgk$ in the integrals, the CohFT gives the ``ancestor" invariants
\begin{equation}\label{ancestors}
\langle \delta_1\bar\psi_1^{a_1},\dots,\delta_k\bar\psi_k^{a_k}
\rangle^\ke_{g,k,\beta}:=\int_{[\Qgke]^{\mathrm{vir}}} \prod_{j=1}^k ev_j^*(\delta_j)\bar\psi_j^{a_j},
\end{equation}
with $\bar\psi_j:=f^*\psi_{j,k}$ and $\psi_{j,k}$ the $\psi$ class at the $j^{\mathrm{th}}$ marking on $\Mgk$.

\subsubsection{genus zero TRR}
For $m\geq 0, 2g-2+k>0$, consider the forgetful map
$$f: {\mathrm {Q}}^\ke_{g,k+m}(\WmodG,\beta)\lra\Mgk.$$
Fix $j\in[k]$. We have the following basic comparison of the classes $\psi_j$ and $f^*\psi_{j,k}$.

\begin{Lemma} \label{psi comparison}
Let $A_j\subset\mathcal{B}(g,k+m)$ be the subset consisting of ordered partitions $\sigma=((0,g), (S_1,S_2))$ with
$j\in S_1$ and $ [k]\setminus\{ j\}\subset S_2$. Then 
\begin{equation}\label{psi comparison eqn}
(\psi_j-f^*\psi_{j,k})\cap [\Qgke]^{\mathrm{vir}}=\sum_{\sigma\in A_j} [D^\ke_\sigma]^{\mathrm{vir}}.
\end{equation}
\end{Lemma}
\begin{proof} The argument in Gromov-Witten theory is not $\ke$-dependent, so will work in general.
\end{proof}

In Gromov-Witten theory, Lemma \ref{psi comparison}
is used to express descendant invariants in terms of ancestors. This can be done for quasimap invariants too, but we will not
deal with it in this paper. Instead, we note that another consequence of the Lemma, the genus zero Topolgical Recursion Relation (TRR) holds for all stability parameters.

\begin{Cor}\label{TRR} (Genus zero TRR) For all $\ke\geq 0+$, all $a_1,a_2,a_3\geq 0$ and all $\delta_1,\delta_2,\delta_3\in H^*(\WmodG,\QQ)$ we have

\begin{equation}
\lla \delta_1\psi_1^{a_1},\delta_2\psi_2^{a_2},\delta_3\psi_3^{a_3}\rra_{0,3}^\ke=
\sum_{i=1}^s \lla \delta_1\psi_1^{a_1-1},\gamma_i\rra_{0,2}^\ke \lla \gamma^i, \delta_2\psi_2^{a_2},\delta_3\psi_3^{a_3}\rra_{0,3}^\ke.
\end{equation}
\end{Cor}

\begin{proof}
Apply Lemma \ref{psi comparison} for $g=0$, $k=3$, and use the fact that $\psi_1$ vanishes on $\overline{M}_{0,3}=\Spec(\CC)$.
\end{proof}

\begin{Rmk}\label{string etc} Gromov-Witten invariants satisfy additional structures encoded in the string, dilaton, and divisor equations.
The reason behind these structures is that moduli of stable maps ($\ke=\infty$) admit forgetful lci morphisms
$$\mathrm{Q}^\infty_{g,k+1}(\WmodG,\beta)\lra \mathrm{Q}^\infty_{g,k}(\WmodG,\beta)$$
for which the virtual classes behave functorially.
These morphisms may fail to exist for stability parameters $\ke\leq 1$. Even when they do exist, the virtual classes often are not compatible with the pull-back.
In particular, for a general target $\WmodG$, the unit cohomology class $\one$ will be the unit for the CohFT only in the Gromov-Witten range of the stability parameter $\ke$.
The appropriate versions of string and dilaton equations for all $\ke$ in the case of semi-positive triples $(W,\G,\theta)$ are discussed later in the paper.
\end{Rmk}
\subsection{Equivariant theory and non-compact targets} \label{equivariant}
Suppose that $W$ admits an additional action by an algebraic torus ${\T}\cong(\CC^*)^n$, which commutes with action of $\G$.
There are induced actions on $[W/\G]$, on $\WmodG$, and on $W/_{\mathrm {aff}}\G$. The projective morphism $\WmodG\lra W/_{\mathrm {aff}}\G$ is $\T$-equivariant. Further, there are also induced 
$\T$-actions on the moduli spaces $\Qgke$ (and on the graph spaces recalled in \S3 below). 

We will always assume in this paper
that the locus of $\T$-fixed points in the affine quotient $W/_{\mathrm {aff}}\G$ is
proper. This assumption is automatic if $\WmodG$ is projective and holds for the natural torus actions on all interesting examples of non-compact targets, such as quasiprojective toric varieties, 
total spaces of bundles over projective quotients, and Nakajima quiver varieties.
It follows that the $\T$-fixed loci in $\WmodG$ and in all the moduli spaces of stable quasimaps are also proper.  
In this situation we get a $\T$-equivariant version of the CohFT, see e.g. \S6.3 in \cite{CKM}
and all results in \S\ref{projective} (which may be viewed as corresponding to the zero-dimensional torus $\T =\{1\}$) are valid in this setting. 

Precisely, let
$$\QQ[\lambda_1,\dots,\lambda_n]=H^*_{\T}(\Spec(\CC),\QQ)$$
and 
$$K:=\QQ(\lambda_1,\dots,\lambda_n)=H^*_{\T,\mathrm{loc}}(\Spec(\CC),\QQ)$$
be the equivariant cohomology, respectively the localized equivariant cohomology of a point. The cohomology ring of $\WmodG$ is replaced by the localized equivariant cohomology
$$H^*_{\T,\mathrm{loc}}(\WmodG,\QQ)=H^*_{\T}(\WmodG,\QQ)\otimes_{\QQ[\lambda_1,\dots,\lambda_n]} K,$$
while the Novikov ring is now $\Lambda=K[[q]]$. 

The pairing is defined by the localization formula
$$\lan\delta,\gamma\ran=\int_{\WmodG^{\T}} \frac{i^*(\delta\gamma)}{\mathrm{e}(N)},$$
where $i:\WmodG^{\T}\hookrightarrow \WmodG$  is the inclusion of the fixed point locus and $\mathrm{e}(N)$ is
the $\T$-equivariant Euler class of the normal bundle. Similarly, the $\T$-equivariant $\ke$-quasimap invariants are defined by the virtual localization formula,
\begin{equation}\label{equiv invariants}
\langle \delta_1\psi_1^{a_1},\dots,\delta_k\psi_k^{a_k}
\rangle^\ke_{g,k,\beta}:=\int_{[(\QmapWe)^{\T}]^{\mathrm{vir}}}\frac{i^*(\prod_j  ev_j^*(\delta_j)\psi_j^{a_j})}{\mathrm{e}(N^{\mathrm{vir}})},
\end{equation}
with $i:(\QmapWe)^{\T} \hookrightarrow\QmapWe$.
Both the pairing and the invariants take values in the field $K$ (or in $\Lambda$, if we take the insertions from $H^*_{\T,\mathrm{loc}}(\WmodG,\Lambda)$).

If $\WmodG$ is {\it projective} and we take all insertions in the non-localized equivariant cohomology $H^*_{\T}(\WmodG,\QQ)$, then the 
$\T$-equivariant invariants
may be defined without localization
and take values in the ring $\QQ[\lambda_1,\dots,\lambda_n]$. Upon specializing $\lambda_1=\dots=\lambda_n=0$ we recover the non-equivariant theory from the previous subsection.
\begin{Rmk} The properness of the moduli spaces over the affine quotient implies immediately that the evaluation maps are proper, so the push-forward $(ev_i)_*$
is well-defined for all targets. The invariant \eqref{equiv invariants} may then be defined as the pairing
\begin{equation}\label{push-forward by ev}
\left\lan \delta_1, (ev_1)_*\left ([\QmapWe]^{\mathrm{vir}}\cap \prod_{j=2}^k  ev_j^*(\delta_j)\prod_{j=1}^k\psi_j^{a_j}\right)\right \ran .
\end{equation}

\end{Rmk}

\subsection{Twisted theories} 
A $\G$-representation $E$ is called {\it convex} (respectively, {\it concave}) if the equivariant vector bundle $W\times E$ on $W$ is generated by 
$\G$-equivariant global sections (respectively, it has no non-zero $\G$-equivariant global sections).
Vector bundles $\underline{E}:= W^s\times_\G E$ on $\WmodG$  induced by representations will be called {\it homogeneous}.
We fix a $1$-dimensional torus ${\bf U}\cong \CC^*$, acting trivially on $W$ and by multiplication on $E$.

Each choice of a representation $E$ and of an invertible multiplicative ${\bf U}$-equivariant
characteristic class $c$ determines an $(E,c)$-twisted $\ke$-quasimap CohFT, see \S6.2 of \cite{CKM}. (We restrict here to bundles arising from representations for simplicity,
but general $\G$-equivariant bundles on $W$ may be considered.)
These generalize to
arbitray stability parameter $\ke$ the twisted Gromov-Witten CohFT's studied by Coates and Givental, \cite{CG}.  

It is shown in \cite{CKM}, and explained again in \S7.2.1 of \cite{CK0}, that when $W$ is smooth, the representation $E$ is convex, and the class $c$ is the (equivariant) Euler class, the resulting
{\it genus zero} twisted $\ke$-quasimap theory recovers (most of\footnote{In the present context,``most of" means that the primary insertions are pulled-back from the ambient $\WmodG$, and that the curve classes 
$\beta$ are those corresponding to $(W,\G,\theta)$.}) the $\ke$-quasimap theory of the zero locus $Z/\!\!/\G\subset \WmodG$ of a regular section of $\underline{E}$
(after specializing the invariants at $\lambda =0$, where $\lambda$ is the equivariant parameter for ${\bf U}$).
For example, all Calabi-Yau complete intersections in toric varieties are covered by this construction, but there are many more cases with indecomposable bundle $\underline{E}$
when the group $\G$ is non-abelian.

When twisting by the {\it inverse} Euler class of a concave representation $E$, the resulting theory coincides with the (untwisted) ${\bf U}$-equivariant theory of the total space of the bundle 
$\underline{E}$ over $\WmodG$,
viewed as the GIT quotient $(W\times E)/\!\!/\G$, in {\it all genera}, see Example 2.8.5 and \S7.2.2 in \cite{CK0}. 
If the base $\WmodG$ is projective, one can specialize 
the ($\beta\neq 0$) invariants
at $\lambda =0$.  
The typical examples we have in mind here occur as follows: take a projective Fano triple $(W,\G,\theta)$, with $W$ a vector space, and take $E=\det (W)^\vee$. These are {\it local Calabi-Yau
targets}, i.e., the total space of the canonical bundle of a Fano GIT quotient.

Unless specified otherwise, whenever we talk about twisting by $E$ in this paper, it should be understood as twisting by the Euler class.

\section{Genus zero theory of semi-positive targets}
Recall from \cite{CK0} that a triple $(W,\G,\theta)$ is called {\it semi-positive} if $$\beta(\det(T_W))\geq 0$$ for every 
$\theta$-effective class
$\beta\in {\mathrm {Eff}}(W,\G,\theta)$. 

The $\ke$-wall-crossing for genus zero invariants with descendant insertions at one point and any number of primary insertions is treated in detail in \cite{CK0} . 
In this section we recall first what those results say for semi-positive targets, then
we discuss the extension to the full genus zero descendant theories in that case. 

From now on, we will assume that $W$ has an action by a torus $\T$ satisfying the assumptions in \S\ref{equivariant} and will consider the $\T$-equivariant $\ke$-quasimap theories (the 
torus is allowed to be trivial if $\WmodG$ is projective). We write simply $H^*(\WmodG)$ for the appropriate $\T$-equivariant (localized) cohomology group.

In addition, some of the results in this section are stated for both untwisted and twisted theories of $\WmodG$. All arguments we give are identical whether the twisting is present or not. 
Hence we do not provide separate proofs, nor do we include the twisting in the notation for brackets, double brackets etc. 

\subsection{Summary of results from \cite{CK0}}

\subsubsection{Graph spaces, $J^\ke$-functions, and $S^\ke$-operators}
 Let $0+\leq\ke$ be a stability parameter. For each $k\geq 0$, we have the {\it graph space} $\QGraphe$, see \cite{CKM}, \S 7.2 and \cite{CK0}, \S 2.6. It is 
 the moduli space of genus zero, $k$-pointed, $\ke$-stable quasimaps whose domain curve contains
 an irreducible component which is a parametrized $\PP^1$. The $\CC^*$-action on the parametrized component lifts to an
 action on $\QGraphe$. The fixed loci for this action are described e.g., in \S4 of \cite{CK0} and their geometry has played an important role in 
 the study of genus zero $\ke$-wall-crossings in quasimap theory undertaken in \cite{CK0} (they will appear later in this paper, in the proof of Theorem \ref{Main Thm}).
 
The (big) $J^\ke$-function is defined as a formal sum over all $k\geq0$ and $\beta\in {\mathrm {Eff}}(W,\G,\theta)$
of localization residues over certain distinguished components of the $\CC^*$-fixed loci in graph spaces,
pushed-forward to $\WmodG$ by evaluation maps, see Definition 5.1.1 in \cite{CK0}.
It takes the form (for arbitrary targets)
\begin{equation}
J^\ke(q,t,z)=\one+\frac{t}{z}+\sum_{i=1}^s \gamma_i \sum _{\beta \neq 0, \beta(L_\theta)\leq 1/\ke}
q^\beta J^\ke_{i,\beta}(z)+ \sum_{i=1}^s\gamma_i\lla \frac{\gamma^i}{z(z-\psi)}\rra^\ke_{0,1}(t),
\end{equation}
with $z$ the generator of the $\CC^*$-equivariant cohomology of $\Spec(\CC)$ and $t=\sum_i t_{0i}\gamma_i \in H^*(\WmodG)$. 
The first three summands account for the missing unstable terms in the double bracket.

For $\ke=\infty$ the double sum disappears and we obtain Givental's big $J$-function in Gromov-Witten theory, with asymptotic expansion $\one +\frac{t}{z}+O(1/z^2)$.
For $\ke=0+$ we use the notation $I(q,t,z):=J^{0+}(q,t,z)$. 
For intermediate values $\ke\in(0,1]$, the double sum is a 
finite $q$-truncation of the infinite double sum in the $I$-function.

Define the {\it small} $J^\ke$-function by restriction to $t=0$,
\begin{equation}\label{small J}
J^\ke_{sm}(q,z):=J^\ke(q,0,z).
\end{equation}
In particular, we have the small $I$-function 
\begin{equation}\label{small I}
I_{sm}(q,z):=I(q,0,z)= \one+\sum_{i=1}^s \gamma_i \sum _{\beta \neq 0}q^\beta I_{i,\beta}(z).
\end{equation}
The double bracket vanishes at $t=0$, since the spaces $\mathrm{Q}^{0+}_{0,1}(\WmodG,\beta)$ are empty for all $\beta$. The terms $\sum_i\gamma_i I_{i,\beta}(z)$ are obtained from
residues on the {\it unpointed} graph spaces $QG^{0+}_{0,0,\beta}(\WmodG)$. These residues have been explicitly calculated in (almost) all interesting examples, giving closed formulas for
the small $I$-function.

Also define, for any $\gamma\in H^*(\WmodG)$,
\begin{equation}\label{S operator}
S^\ke_t(\gamma):= \sum_{i=1}^s\gamma_i \lla \frac{\gamma^i}{z-\psi},\gamma\rra^\ke_{0,2}(t).
\end{equation}
The convention  
$$\lan \frac{\gamma^i}{z-\psi},\gamma\ran^\ke_{0,2,0}=\lan\gamma^i,\gamma\ran$$ 
is made for the unstable term (corresponding to $m=0,\beta=0$) in $\lla ... \rra^\ke_{0,2}$. 
Directly from definitions, for every $i=1,\dots, s$,
\begin{equation}\label{derivatives}
z\frac{\partial}{\partial t_{0i}} J^\ke(t)= S^\ke_t(\gamma_i).
\end{equation}
It is shown in \cite{CK0} that the formula \eqref{S operator} defines a family (with parameter $t$) of symplectic transformations on the symplectic space 
$\mathcal{H}=H^*(\WmodG,\Lambda)\{\!\{z,z^{-1}\}\!\}$ appearing in Givental's formalism of Gromov-Witten theory, \cite{Givental-symplectic}.

The most general $\ke$-wall-crossing formula in genus zero applies to the operators
$S^\ke_t$, see Conjecture 6.1.1 and Theorem 7.3.1 in \cite {CK0}. We state here a special case, as formulated
in \cite{CK0}, Theorem 1.2.2. 

\begin{Thm}\label{mirror} Assume that the $\bT$-action on $\WmodG$ has isolated fixed points. Then for every $\ke\geq 0+$ we have
\begin{equation}
S^\ke_t(\one)=S^\infty_{\tau^\ke(t)}(\one)
\end{equation}
where the (invertible) transformation $\tau^\ke(t)$ is the following series of primary $\ke$-quasimap invariants
\begin{equation}\label{mirror map}
\begin{split}
&\tau^\ke(t)=\sum_{i=1}^s \gamma_i\lla \gamma^i,\one\rra(t)-\one=\\
&t+\sum_{i=1}^s\gamma_i\sum_{\beta\neq 0}\sum_{m\geq 0}\frac{q^\beta}{m!}\lan\gamma^i,\one,t,\dots,t\ran^\ke_{0,2+m,\beta}.
\end{split}
\end{equation}
Moreover, the same is true for $E$-twisted theories, where $E$ is a convex $\G$-representation.
\end{Thm} 
No positivity assumptions are made in Theorem \ref{mirror} on $(W,\G,\theta)$, or on $(W,E,\G,\theta)$ in the twisted case.
Of course, the statement is conjectured to hold irrespective of the existence of a torus action with isolated fixed points. As already explained, the part of the Theorem involving twisted theories 
covers such targets, since it concerns the genus zero $\ke$-quasimap theory of the zero locus of a regular section of the bundle $\underline{E}$ and this zero-locus is generally not $\T$ -invariant.

\subsubsection{The semi-positive case}
When $(W,\G,\theta)$ is semi-positive, no positive powers of $z$ appear in $J^\ke$. Define $J^\ke_0(q)$ and $J^\ke_1(q)$ from the asymptotic expansion
\begin{equation}\label{z-expansion} J^\ke(q,t,z)=J^\ke_0(q)\one + (t+J^\ke_1(q))\frac{1}{z}+O\left(\frac{1}{z^2}\right).
\end{equation}
In particular, we have
$q$-series $I_0(q)$ and $I_1(q)$, with
$$I_{sm}(q,z)=I_0(q)\one+I_1(q)\frac{1}{z}+O\left(\frac{1}{z^2}\right).$$
They satisfy $I_0(q)=1+O(q)\in\Lambda$ and $I_1\in {\mathfrak{m}}H^{\leq 2}(\WmodG,\Lambda)$. For $\ke >0$,
the coefficients $J^\ke_0(q)$ and $J^\ke_1(q)$ are polynomial truncations of the series $I_0$ and $ I_1$,
\begin{equation}\label{truncation}
J^\ke_0(q)=I_0(q)\; (\mathrm{mod}\;\mathfrak{a}_\ke),\;\;\;\;  J^\ke_1(q)=I_1(q)\; (\mathrm{mod}\;\mathfrak{a}_\ke),
\end{equation}
where $\mathfrak{a}_\ke$ is the ideal in the Novikov ring generated by $\{q^\beta\; |\: \beta(L_\theta) >\frac{1}{\ke}\}$.

The Proposition below collects the results for semi-positive targets from \cite{CK0}.

\begin{Prop}\label{summary} Let $(W,\G,\theta)$ be semi-positive and let $\ke\geq 0+$ arbitrary. Then

$(i)$ The $J$-function and the $S$-operator are related by
$$S^\ke_t(\one)=\frac{J^\ke(q,t,z)}{J^\ke_0(q)}.$$

$(ii)$ The transformation \eqref{mirror map} satisfies
$$\tau_\ke(t)=\frac{t+J^\ke_1(q)}{J^\ke_0(q)}.$$
In particular, 
\begin{equation}\label{mirror map semi}
\sum_{i=1}^s\gamma_i\sum_{\beta\neq 0} q^\beta\lan\gamma^i,\one\ran_{0,2,\beta}^\ke=\frac{J^\ke_1(q)}{J^\ke_0(q)}.
\end{equation}

$(iii)$ If the $\T$-action on $\WmodG$ has isolated fixed points, then
$$J^\infty\left(q, \frac{t+J^\ke_1(q)}{J^\ke_0(q)}, z\right)=\frac{J^\ke(q,t,z)}{J^\ke_0(q)}.$$
The same is true for $E$-twisted theories on $\WmodG$, where $E$ is a convex $\G$-representation such that $\beta(\det(T_W))-\beta(W\times \det(E))\geq 0$ for all $\theta$-effective $\beta$.

$(iv)$ Under the same assumption as in $(iii)$, for $n\geq 2$, $a\geq 0$, and $i_1,\dots, i_n \in \{1,\dots, s\}$, 
\begin{equation}\label{1-pt Taylor}
\begin{split}
&\;\;\;\;\;\;\;\;\; (J_0^\ke(q))^{n-2}\sum_{\beta\geq 0} q^\beta \langle \gamma_{i_1}\psi_1^{a},\gamma_{i_2},\dots,\gamma_{i_n}
\rangle^\ke_{0,n,\beta}=\\
 & \sum_{\beta \geq 0} q^\beta\sum_{m\geq 0} \frac{1}{m!}\left\lan \gamma_{i_1}\psi_1^{a},\gamma_{i_2},\dots,\gamma_{i_n}, \frac{J^\ke_1(q)}{J^\ke_0(q)},\dots, \frac{J^\ke_1(q)}{J^\ke_0(q)}
\right\ran_{0,n+m,\beta}^\infty .
\end{split}
\end{equation}
\end{Prop}

\begin{proof} For parts $(i)$ and $(ii)$, see \cite{CK0}, Corollary 5.5.3. Part $(iii)$ follows from $(i)$, $(ii)$, and Theorem \ref{mirror}. Part $(iv)$ is obtained by matching
Taylor coefficients in $(iii)$, see Corollary 1.5.2 in \cite{CK0}. \end{proof}
\begin{Rmk}\label{degree one}
Equation \eqref{1-pt Taylor} proves Theorem \ref{genus zero intro} in the case when only one of the insertions is descendant and the rest are primaries. 
In view of $(iii)$, we may extend it to $n=1$ by interpreting the left-hand side as the coefficient of $\gamma^{i_1}/z^{a+2}$ in $J^\ke_{sm}$. With this interpretation we may therefore also extend
the genus zero potential $F_0^\ke$ to include the missing terms in $q^\beta t_{ji}$ for $\beta(L_\theta)\leq 1/\ke$. The genus zero case of \eqref{GW to epsilon} (and Theorem \ref{genus zero intro})
will then be viewed as a matching of potentials up to an additive constant. 
\end{Rmk}
\begin{Rmk}\label{unit} Part $(i)$ contains the statement
\begin{equation}\label{string equation 1} 
\lla \frac{\gamma}{z-\psi},\delta,J^\ke_0(q)\one\rra^\ke_{0,3} (t)=\frac{1}{z}\lla \frac{\gamma}{z-\psi},\delta\rra^\ke_{0,2} (t),
\end{equation}
which says that $J^\ke_0(q)\one$ satisfies the string equation for one-point descendants in $\ke$-quasimap theory of a semi-positive target. 
In particular, the same class is the unit for the $\ke$-quantum multiplication (cf. Corollary 5.5.4 in \cite{CK0}). 
\end{Rmk}
\begin{Rmk} As explained in \cite{CK0}, Remark 6.2.2, parts $(iii)$ and $(iv)$ of Proposition \ref{summary} generalize the genus zero toric mirror theorems of \cite{Givental}.

\end{Rmk}
 
 \subsection{Two descendant insertions}
 
 Denote by $[\Delta]$ the cohomology class of the diagonal 
 $$[\Delta] =\sum_{i=1}^s\gamma_i\ot \gamma^i\in H^*(\WmodG)\ot H^*(\WmodG).$$
Let $z,w$ be formal variables and define

\begin{equation} \label{V operator}
V ^{\ke}_t (z, w)  :=\sum _{i,j=1}^s \gamma _i \ot \gamma _j \lla \frac{\gamma ^i}{z-\psi } , \frac{\gamma ^j}{w-\psi }\rra^\ke_{0,2}(t) .
\end{equation}
The convention $$\sum_{i,j=1}^s \gamma_i\ot\gamma_j\lan \frac{\gamma ^i}{z-\psi } , \frac{\gamma ^j}{w-\psi }\ran^\ke_{0,2,0}=\frac{[\Delta]}{z+w}$$
is made for the unstable term in the double bracket.
We have
$$V ^{\ke}_t (z, w) - \frac{[\Delta]}{z+w}\in H^*(\WmodG)\ot H^*(\WmodG)[[q, \{t_{0j}\}, 1/z, 1/w]]. $$

\begin{Thm}\label{Two_Pointed} For arbitrary GIT targets $\WmodG$,

    \begin{equation}\label{V and S}  V^{\ke}_t = \frac{S^{\ke} _t (z)\ot S^{\ke} _t(w) ([\Delta] )}{z+w}. \end{equation}
    
\end{Thm}

\begin{Rmk}\label{genus zero 2 pointed} Combining Theorem \ref{Two_Pointed} with Proposition \ref{summary}$(iii)$ and \eqref{derivatives} proves
Theorem \ref{genus zero intro} in the case when two of the insertions are descendant and the other insertions are primary.
A very special case (twisted ($0+$)-theory of $\PP^n$, with two descendant and no primary insertions) has also been proved by different methods in \cite{Zinger}.
\end{Rmk}

In Gromov-Witten theory ($\ke=\infty$), the statement in Theorem \ref{Two_Pointed} is well-known and its proof follows immediately from the WDVV and string equations, see
\cite{Givental-elliptic}, item $(4)$ on p.117. In quasimap theory the string equation is apriori missing, so this proof will not work. Instead, we provide a
localization argument which is a variant  ``with two equivariant parameters" of
the proofs of Proposition 5.3.1 and Theorem 5.4.1 in \cite{CK0}. 

Before going into details, we note first that the usual argument shows that \eqref{V and S} and the string equation for invariants with two descendant insertions are equivalent,
in the presence of WDVV {\it and} the string equation for invariants with at most one descendant insertion.
Hence, using \eqref{string equation 1} and Theorem \ref{Two_Pointed} we obtain again that, in the semi-positive case, $J_0 ^\ke \one$ satisfies the string equation for two-point descendants:

\begin{Cor}\label{unit two-pointed} For arbitrary  semi-positive $(W,\G,\theta)$ 
, 
\begin{equation}\label{string 2} \lla J_0 ^\ke \one , \frac{\gamma }{z-\psi }, \frac{\delta}{w-\psi } \rra ^\ke _{0, 3} (t)=
\frac{z+w}{zw} \lla  \frac{\gamma }{z-\psi }, \frac{\delta}{w-\psi } \rra _{0, 2} ^\ke (t) \end{equation}
where the unstable term in the right-hand side double bracket is defined to be $\frac{\lan\gamma, \delta\ran}{z+w}$.
\end{Cor}

\subsubsection{The $\Atwo$-graph space}

We will require a version of graph spaces for which the domain curve has {\it two} parametrized components. To construct it, consider the crepant resolution of
the $A_2$-singularity 
$$\pi:\Atwo\lra Y_0=\CC^2/\ZZ_3.$$  
$\Atwo$ is a smooth quasi-projective surface and the
exceptional set of $\pi$ is a nodal curve $D=E_1\cup E_2$. The two components are rational $(-2)$-curves meeting in a point. 

Since $\Atwo$ is identified with the $\ZZ _3$-Hilbert scheme of $\CC ^2$, it has a natural action by a two-dimensional torus $\bS\cong(\CC^*)^2$, induced from the
standard $\bS$-action on $\CC ^2$. There are exactly three $\bS$-fixed points in $\Atwo$: the node on $D$, and one additional fixed point on each component of $D$. 
We denote the fixed points $p_0$, $p_n$, $p_\infty$, with $p_0\in E_1$, $p_n=E_1\cap E_2$, and $p_\infty\in E_2$. There are two compact  one-dimensional $\bS$-orbit closures, namely
$E_1$ and $E_2$, and two noncompact ones, $D_0$ passing through $p_0$ and $D_\infty$ passing through $p_\infty$.

Let $H^*_{\bS} (pt ) = \QQ [s_1, s_2] $, so that $s_1, s_2$ are the equivariant parameters. We denote by $z$, respectively by $w$, the $\bS$-weights on $E_1$, respectively on $E_2$ at the
node $p_n$. We have $z=2s_1-s_2$ and $w=2s_2-s_1$. The weights at the other fixed points are $-z$ on $E_1$ and $2z+w$ on $D_0$ at $p_0$, and $-w$ on $E_2$ and $2w+z$ on $D_\infty$
at $p_\infty$. Note that the sum at each fixed point is $z+w=s_1+s_2$, reflecting the fact that $\Atwo$ has a holomorphic symplectic form induced by the standard form on $\CC^2$.

Any nonconstant map from a projective curve to $\Atwo$ must factor through $D$.
Fix $\beta\in \Hom_\ZZ(\Pic^\G(W),\ZZ)$ and consider the moduli stack 
\[ \mathfrak{M}_{0, k}([W/\G]\times \Atwo, (\beta, 1,1)) \]
parametrizing maps from $k$-pointed, genus zero curves to $[W/\G]\times \Atwo$, of class $(\beta, 1,1)$.
A geometric point in this stack is a tuple
$$ ((C,x_1,\dots, x_k), f, \varphi),$$
with $(C,x_1,\dots, x_k)$ a prestable curve of genus zero, $f:C\lra [W/\G]$ a map of class $\beta$, and
$\varphi: C\lra \Atwo$ a regular map such that $\varphi_*[C]=[D]$.
In particular, the domain curve must have two distinguished irreducible components $C_1$ and $C_2$ such that $\varphi$ maps $C_i$ isomorphically onto
$E_i$ and contracts all other components of $C$.

Next, for each $0+\leq\ke\leq \infty$ we introduce $\ke$-stability in almost the same way as for the usual graph spaces, see \cite{CKM}, Definition 7.2.1 and
\cite{CK0}, \S2.6. The only difference is that the ampleness part of the stability condition does not involve {\it either} of the two distinguished components in the domain.
Precisely, we require that
$$\omega _{C'} (\sum z_i +\sum y_j) \otimes f^*{L}_{\theta} ^{\otimes \ke} $$
is ample on $C'$, where $C'$ is the closure of ${C\setminus (C_1\cup C_2)}$, $z_i$ are the markings on $C'$ and $y_j$ are the nodes $C'\cap (C_1\cup C_2)$.

Imposing the $\ke$-stability condition determines an open substack
\begin{equation}\label{A2-graph space} \mathrm{Q}^{\ke}_{0, k} (X, \beta;  \Atwo) \end{equation}
of $\mathfrak{M}_{0, k}([W/\G]\times \Atwo, (\beta, 1,1))$, which we will call the $\Atwo$-{\it graph space}.  
The $\bS$-action on $\Atwo$ induces a $\bS$-action on $\mathrm{Q}^{\ke}_{0, k} (\WmodG, \beta;  \Atwo)$. Recall that we also have a $\T$-action on $W$; it lifts as well to an action on
the $\Atwo$-graph space. These two actions commute, so we have a $\T\times\bS$-action.

\begin{Prop}
The moduli space $\mathrm{Q}^{\ke}_{0, k} (\WmodG, \beta;  \Atwo)$ has the following properties.
\begin{enumerate}
\item It is defined for all $\ke\geq 0+$, $k\geq 0$, and $\beta\in \mathrm{Eff}(W,\G,\theta)$.
\item It is a separated Deligne-Mumford stack of finite type.
\item It has a natural proper map to the affine quotient $W/_{\mathrm {aff}}\G$. In particular, it is proper when $\WmodG$ is projective.
\item It carries a natural $\T\times \bS$-equivariant perfect obstruction theory.
\end{enumerate}
\end{Prop}
\begin{proof} Part $(1)$ is obvious. Parts $(2)-(4)$ follow in a straightforward manner using the arguments in \cite{CKM}. 

For the properness in part $(3)$ we use in addition that 
$$(f,\varphi):C\lra [W/\G]\times \Atwo$$ factors through $[W/\G]\times D$.

In part $(4)$, the relative obstruction theory over the smooth stack $\mathfrak{M}_{0,k}$ of prestable curves is the direct sum of the relative obstruction theories for quasimaps to $[W/\G]$
and for maps to $\Atwo$. The $\T$-equivariance comes from the first summand, while the $\bS$-equivariance comes from the second summand. 
The absolute obstruction theory is obtained as usual from the relative one via a distinguished triangle in the derived category.
\end{proof}

\subsubsection{Proof of Theorem \ref{Two_Pointed}}

The $\Atwo$-graph space comes with evaluation maps
$$\tilde{ev}_i : \mathrm{Q}^{\ke}_{0, k} (\WmodG, \beta;  \Atwo)\lra \WmodG\times\Atwo.$$ 
We write $\delta v$ for the $(\T\times \bS)$-equivariant cohomology class $\delta\ot v\in H^*(\WmodG)\ot H^*_{\bS}(\Atwo)$
and simply $\delta$ for $\delta\ot 1$. (Recall that $H^*(\WmodG)$ denotes $\T$-equivariant cohomology, localized if needed, but we supressed $\T$ from the notation.)
Define $(\ke, \Atwo)$-double brackets by
\begin{align*}
& \lla \delta_1 v_1,\dots ,\delta_r v_r\rra_{0,r}^{(\ke, \Atwo)}(t)=\\ 
&\sum _{\beta, k\geq 0 }\frac{q^\beta}{k!}
\int_{[\mathrm{Q}^{\ke}_{0, r+k} (\WmodG, \beta;  \Atwo)]^{\mathrm{vir}}}\prod_{l=1}^r\tilde{ev}_l^*(\delta _l v_l)
\prod_{m=r+1}^{r+k}\tilde{ev}_m^*(t).\end{align*}
The integral in the above formula is understood as $\T\times\bS$-equivariant push-forward to a point. By properness in the $\Atwo$ direction, the integrals are well defined {\it without}
localization for the $\bS$-action, so the double bracket has no poles at the equivariant parameters for $\bS$, i.e., it takes values in 
$$\Lambda[[\{t_{0j}\}]][[s_1,s_2]]=\Lambda[[\{t_{0j}\}]][[z,w]].$$

We also use the notation $D_0,D_\infty\in H^*_{\bS}(\Atwo,\QQ)$ for the $\bS$-equivariant divisor classes of the two noncompact one-dimensional $\bS$-orbits in $\Atwo$ 
from the previous subsection.
Their restrictions at the fixed points on $\Atwo$ are
$$D_0|_{p_0}=-z,\; D_0|_{p_n}=D_0|_{p_\infty}=0,\;\;\; D_\infty |_{p_0}=D_\infty |_{p_n}=0,\; D_\infty |_{p_\infty}=-w .$$ 
Now consider the generating series
\begin{equation*}
R^\ke : =  \sum_{i,j=1}^s \gamma _i \ot \gamma _j \lla \gamma ^i D_0, \gamma^j D_\infty \rra_{0,2}^{(\ke, \Atwo)}(t).
\end{equation*}
By definition, it is an element of $H^*(\WmodG)\ot H^*(\WmodG)[[q, \{t_{0j}\}, z, w]]$. The coefficient of each monomial $q^\beta t_{01}^{\alpha_1}\dots t_{0s}^{\alpha_s}$ is 
{\it polynomial} in $z$ and $w$. One calculates
\begin{equation}
R^\ke =(z+w) [\Delta] + {\text{higher\; order\; terms\; in\; }} q, t.
\end{equation}
This is easily seen since we have  $$\mathrm{Q}^{\ke}_{0, 2} (\WmodG, 0;  \Atwo)\cong \WmodG\times \overline{M}_{0,2}(\Atwo, (1,1)),$$
therefore the term we want to evaluate is
$[\Delta]$ times the equivariant $2$-point Gromov-Witten invariant $\lan D_0, D_\infty\ran_{0,2, (1,1)}^{\Atwo}$ of $\Atwo$. 
By a simple localization computation this Gromov-Witten invariant equals $z+w$.

In fact, we can evaluate the full series $R^\ke$ by virtual localization for the $\bS$-{\it action only}.
The description of the $\bS$-fixed loci on the $\Atwo$-graph spaces $\mathrm{Q}^{\ke}_{0, 2+k} (\WmodG, \beta;  \Atwo)$, 
together with the fixed and moving parts of the obstruction theory is similar to the description in the case of ``usual" graph spaces,
for which details can be found in \S4 of \cite{CK0}. 

Let $ ((C,x_1,x_2,\dots, x_{k+2}), f, \varphi)$ be $\bS$-fixed. Then the map $f$ must contract the distinguished components $C_1$ and $C_2$ to the same point in $\WmodG$. The rest of the
curve is contracted by $\varphi$ to the fixed points in $\Atwo$. It follows that
$\varphi^{-1}(\{p_0,p_n,p_\infty\})$ has exactly three connected components, denoted $C_0$, $C_n$, and $C_\infty$.

Due to the insertions of $D_0$ and $D_\infty$, the only components of the $\bS$-fixed locus that contribute to the localization computation are those for which $\varphi(x_1)=p_0$ and
$\varphi(x_2)=p_\infty$. Each such
component, denoted $F_{k_0,k_n,k_\infty}^{\beta_0,\beta_n,\beta_\infty}$, corresponds to a pair of ordered splittings $k=k_0+k_n+k_\infty$ and $\beta=\beta_0+\beta_n+\beta_\infty$, 
with $k_i\geq 0$ and
$\beta_i\in \mathrm{Eff}(W,\G,\theta)$. It is isomorphic to

\begin{equation}\label{S-fixed} 
\mathrm{Q}^\ke_{0,1+k_0\cup\bullet}(\WmodG,\beta_0)\times_{\WmodG} \mathrm{Q}^\ke_{0,k_n\cup\{\bullet,\star\}}(\WmodG,\beta_n)\times_{\WmodG}
\mathrm{Q}^\ke_{0,1+k_\infty\cup\star}(\WmodG,\beta_\infty),
\end{equation}
where the first fiber product is with respect to the evaluation maps $ev_\bullet$, the second fiber product is with respect to $ev_\star$.
The unstable cases $(k_i,\beta_i)=(0,0)$ are included in the above description by the conventions
$$\mathrm{Q}^\ke_{0,1\cup\bullet}(\WmodG,0)=\mathrm{Q}^\ke_{0,\{\bullet,\star\}}(\WmodG,0)=\mathrm{Q}^\ke_{0,1\cup\star}(\WmodG,0):=\WmodG,$$ 
$$ev_\bullet=ev_\star := id_{\WmodG} .
$$
The domain curves for the factors in \eqref{S-fixed} correspond to the three connected components $C_0$, $C_n$, and $C_\infty$.

The virtual class $[F_{k_0,k_n,k_\infty}^{\beta_0,\beta_n,\beta_\infty}]^{\mathrm{vir}}$, determined by the fixed part of the (absolute) obstruction theory, is equal to the refined Gysin 
pull-back of the virtual classes
on the factors in \eqref{S-fixed} by the appropriate diagonal map. 

The Euler class of virtual normal bundle is determined by the moving parts of $H^\bullet(C, \varphi^*T_{\Atwo})$, and of the deformations and automorphisms of the domain curve. It
is easily obtained from the normalization sequence for $$C=C_0\cup C_1\cup C_n\cup C_2\cup C_\infty$$ 
by a standard calculation. After all cancellations, its inverse has the form
\begin{equation}\label{S-normal}
\frac{1}{\mathrm{e}(N^{\mathrm{vir}})}=\frac{(z+w)^2}{zw\; \mathrm{cont} (0)\; \mathrm{cont} (n)\; \mathrm{cont} (\infty)},
\end{equation}
where 
\begin{equation*}\label{cont 0}
\mathrm{cont}(0)=\begin{cases}
1, & (k_0,\beta_0)=(0,0),\\
(-z-\psi_\bullet) , & \text{otherwise},
\end{cases}\end{equation*}
\begin{equation*}\label{cont n}
\mathrm{cont}(n)=\begin{cases}
z+w, & (k_n,\beta_n)=(0,0),\\
(z-\psi_\bullet)(w-\psi_\star ), & \text{otherwise},
\end{cases}\end{equation*}
\begin{equation*}\label{cont infty}
\mathrm{cont}(\infty)=\begin{cases}
1, & (k_\infty,\beta_\infty)=(0,0),\\
(-w-\psi_\star) , & \text{otherwise}.
\end{cases}\end{equation*}
For example, the two  factors of $z+w$ in the numerator come from the moving parts of $H^1(C_1,\varphi^*T_{\Atwo})$ and $H^1(C_2,\varphi^*T_{\Atwo})$.

Applying the virtual localization formula gives the factorization 
\begin{equation} \label{factorization 1}
\begin{split}
& \lla \gamma ^i D_0, \gamma^j D_\infty \rra_{0,2}^{(\ke, \Atwo)}(t)= (z+w)^2\times\\
&\sum_{l,m=1}^s \lla\gamma^i,\frac{\gamma_l}{-z-\psi_\bullet}\rra_{0,2}^\ke(t) 
\lla \frac{\gamma^l}{z-\psi_\bullet}, \frac{\gamma^m}{w-\psi_\star}\rra_{0,2}^\ke(t) 
\lla \frac{\gamma_m}{-w-\psi_\star},\gamma^j \rra_{0,2}^\ke(t). 
\end{split}
\end{equation}
The double brackets in the right-hand side include the unstable terms, as defined in \eqref{S operator} and \eqref{V operator}.

Recall now from \cite{CK0}, Proposition 5.3.1, that the unitary property of the $S^\ke$-operator states that its inverse is the operator defined by
\begin{equation*}
(S_t^\ke)^\star(-z)(\gamma)=\sum_i\gamma_i\lla \gamma^i,\frac{\gamma}{-z-\psi}\rra^\ke_{0,2}(t).
\end{equation*}

We then get from \eqref{factorization 1}
\begin{equation}\label{factorization 2} 
R ^\ke=     (z+w)^2 ((S_t^\ke)^\star (-z)\ot (S_t^\ke) ^\star (-w)) ( V_t^\ke (z, w) ).\end{equation}   
One checks immediately that in the right-hand side the terms of total degree zero in $z$ and $w$ cancel out, so
that the only term without a pole is $(z+w)^2\frac{[\Delta]}{z+w}=(z+w)[\Delta]$.
On the other hand,  $R^\ke$ has no poles, since it is a power series in $z,w$. We conclude that 
\begin{equation}\label{R} 
R^\ke = (z+w)[\Delta].
\end{equation}
Now \eqref{R}, \eqref{factorization 2}, and the unitary property imply the formula \eqref{V and S}. Theorem \ref{Two_Pointed} is proven.

\subsection{TRR and the proof of Theorem \ref{genus zero intro}} 

The following Proposition, together with Proposition \ref{summary}$(iii)$ obviously implies Theorem \ref{genus zero intro}.
\begin{Prop} Let $(W,\G,\theta)$ be an arbitrary semi-positive triple and let $n\geq 2$. If
\begin{equation*}
  (J_0^\ke(q))^{n-2}\lla \delta_1\psi_1^{a_1},\dots,\delta_n\psi_n^{a_n}
\rra^\ke_{0,n}(t)=
\lla \delta_1\psi_1^{a_1},\dots,\delta_n\psi_n^{a_n}\rra_{0,n}^\infty\left( \frac{t+J^\ke_1(q)}{J^\ke_0(q)}\right)
\end{equation*}
holds when all except possibly one of the $a_i$'s are equal to zero, then it holds in general. 

\end{Prop}
\begin{proof} It suffices to assume that the cohomology classes $\delta_i$ are
elements in our chosen basis of $H^*(\WmodG)$. 

We use induction on $n$. The base case $n=2$ follows from Theorem \ref{Two_Pointed}, 
as we already pointed out in Remark \ref{genus zero 2 pointed}.  

Let $n\geq 3$. Let $j_1,\dots ,j_n\in\{1,2,\dots, s\}$.
By taking derivatives in the TRR relation of Corollary \ref{TRR} we get 
\begin{equation}\label{n point TRR}
\begin{split}
&\lla \gamma_{j_1}\psi_1^{a_1},\gamma_{j_2}\psi_2^{a_2},\dots ,\gamma_{j_n}\psi_n^{a_n}\rra_{0,n}^\ke =\\
&\sum_{i=1}^s \sum_{S,T} \lla \gamma_{j_1}\psi_1^{a_1-1},(\gamma \psi^a)^{S} ,\gamma_i\rra_{0,2+|S|}^\ke
\lla \gamma^i, \gamma_{j_2}\psi_2^{a_2},\gamma_{j_3}\psi_3^{a_3}, (\gamma\psi^a)^{T}\rra_{0,3+|T|}^\ke.
\end{split}
\end{equation}
The inner sum is over all partitions $S\coprod T=\{4,\dots,n\}$, while 
the notation $(\gamma \psi^a)^{S}$ stands for the insertions $\gamma_{j_l}\psi^{a_l}$ at the appropriate markings, 
with $l$ running in $S$, and likewise for $(\gamma \psi^a)^{T}$. The double brackets are evaluated at ${\bf t}(\psi)$. By linearity, we may take the coefficients $t_i$ in 
${\bf t}(\psi)=t_0+t_1\psi+t_2\psi^2+\dots$ to lie in $H^*(\WmodG,\Lambda)$.

The relation \eqref{n point TRR} holds for all stability parameters. We specialize it at ${\bf t}(\psi)=t_0=t$ for parameter $\ke$, while for parameter $\infty$ we multiply it
by $J_0^\ke(q)^{2-n}$ and then specialize at ${\bf t}(\psi)=\frac{t+J^\ke_1(q)}{J^\ke_0(q)}$.
By the induction assumption the two right-hand sides of the resulting relations are equal, hence the same is true for the two left-hand sides. 
This proves the Proposition, hence Theorem \ref{genus zero intro} as well.
\end{proof}

\subsection{The string and dilaton equations}\label{string-dilaton} We close this section with a discussion of the versions of string and dilaton 
equations that hold for a general stability parameter $\ke$.
\subsubsection{String} For the string equation, we have a completely general result in genus zero: the 
class $ J_0^\ke(q)\one$ satisfies the string equation for the full genus zero descendant $\ke$-stable quasimap theory.

\begin{Prop}\label{genus zero string} Let the
semi-positive triple $(W,\G,\theta)$ be arbitrary. Then for every $\ke\geq 0+$ and $n\geq 2$,
\begin{equation}\label{string equation} \begin{split}
&\sum_{\beta}q^\beta\lan \delta_1\psi_1^{a_1},\dots , \delta_n\psi_n^{a_n}, J_0^\ke(q)\one\ran_{0,n+1,\beta}^\ke= \\
& \sum_\beta q^\beta\sum_{j=1}^n\lan \delta_1\psi_1^{a_1},\dots , \delta_{j-1}\psi_{j-1}^{a_{j-1}}, \delta_j\psi_j^{a_j-1}, \delta_{j+1}\psi_{j+1}^{a_{j+1}},\dots , \delta_n\psi_n^{a_n} \ran_{0,n,\beta}^\ke.
\end{split}\end{equation}
\end{Prop}
\begin{proof} Equation \eqref{string equation} is obtained by setting $t=0$ in
\begin{equation}\label{string 3} \begin{split}
&\lla \delta_1\psi_1^{a_1},\dots , \delta_n\psi_n^{a_n}, J_0^\ke(q)\one\rra_{0,n+1}^\ke(t)=\\
&\lla \delta_1\psi_1^{a_1},\dots , \delta_{j-1}\psi_{j-1}^{a_{j-1}}, \delta_j\psi_j^{a_j-1}, \delta_{j+1}\psi_{j+1}^{a_{j+1}},\dots , \delta_n\psi_n^{a_n} \rra_{0,n}^\ke(t).
\end{split}\end{equation}
The case $n=2$ of \eqref{string 3} is given by Corollary \ref{unit two-pointed}. The case $n\geq 3$ follows then by induction, using the TRR equation \eqref{n point TRR}.
\end{proof}
\begin{Rmk}
It is very easy to see that Conjecture \ref{equivalent} together with the string equation in Gromov-Witten theory
imply \eqref{string equation} for any genus $g$. In fact, we conjecture that 
the $\ke$-CohFT of a semi-positive target is a cohomological field theory with unit $\one_\ke:=J_0^\ke\one$ over the Novikov ring. 
This means explicitly the following: for all $g,k$ with $2g-2+k\geq 1$ and arbitrary cohomology classes $\delta_1,\dots,\delta_k\in H^{2*}(\WmodG,\Lambda)$, we conjecture that the maps
\eqref{Omega} satisfy 
$$\Omega^\ke_{g,k+1}((\ot_{j=1}^k\delta_j)\ot\one_\ke)=p^*\Omega^\ke_{g,k}(\ot_{j=1}^k\delta_j),
$$ where $p:\overline{M}_{g,k+1}\lra\overline{M}_{g,k}$ is the forgetful map. 
We note that it should be possible to prove Conjecture \ref{main conj intro} in the case of {\it semisimple} theories (such as the fully equivariant theories that appear in Theorems \ref{Main Thm}  and \ref{local Grass}) by combining genus zero results as proved in the present paper with the
Givental/Teleman formula \cite{Giv-ss,Teleman} for higher-genus potentials. However, such an approach would require establishing first the above conjecture about the unit of the $\ke$-CohFT.

\end{Rmk}

\subsubsection{Dilaton} 
\begin{Lemma} Assume Conjecture \ref{equivalent} holds for the semi-positive triple $(W,\G,\theta)$ and the stability parameter $\ke\geq 0+$. Then $(J^\ke_0\one)\psi-J^\ke_1$ satisfies the
dilaton equation for $\ke$-stable quasimap theory: 
 
 \begin{equation}\label{dilaton} \begin{split}
&\sum_{\beta}q^\beta\lan \delta_1\psi_1^{a_1},\dots , \delta_n\psi_n^{a_n}, (J_0^\ke(q)\one)\psi-J_1^\ke(q)\ran_{g,n+1,\beta}^\ke= \\
& (2g-2+n)\sum_\beta q^\beta  \lan \delta_1\psi_1^{a_1},\dots , \delta_n\psi_n^{a_n} \ran_{g,n,\beta}^\ke,
\end{split}\end{equation}
the sums over all $\theta$-effective $\beta$ with $2g-2+n+\ke\beta(L_\theta)>0$.
\end{Lemma} 

\begin{proof} Using the dilaton equation in Gromov-Witten theory, this is an elementary calculation which is left to the reader.
\end{proof}

\begin{Cor}\label{genus zero dilaton} Under the same assumptions as in Theorem \ref{genus zero intro}, the dilaton equation \eqref{dilaton} 
holds in genus zero.
\end{Cor}

\begin{Rmk}For $g=0$ and $n=2$, by using Corollary \ref{TRR} (at ${\bf t}(\psi)=0$) and equation \eqref{mirror map semi}, it is easy to see that the $\ke$-dilaton equation
$$\sum_{\beta}q^\beta\lan \delta_1\psi_1^{a_1},\delta_2\psi_2^{a_2}, (J_0^\ke(q)\one)\psi-J_1^\ke(q)\ran_{0,3,\beta}^\ke=0
$$
holds without any additional assumptions on the semi-positive triple $(W,\G,\theta)$. \end{Rmk}

\section{Virtual classes and $\ke$-wall-crossing} 
\subsection{Overview} The semi-positive GIT targets in Theorems \ref{Main Thm} and \ref{local Grass} share the common
feature that their $I$-functions satisfy $I_0(q)=1$, and hence $J_0^\ke (q)=1$ for all $\ke$. 
For semi-positive toric varieties this can be checked using the explicit formula
for the small $I$-function, as given in \cite{Givental}, see Lemma \ref{I_0=1} later in the paper,
while for local Calabi-Yau targets (as in Theorem \ref{local Grass}) it is an easy general fact, see Remark 5.5.6 in \cite{CK0}.
Therefore, the statement to be proved reduces to the identification of potentials after 
a shift by $J_1^\ke(q)$,
\begin{equation}
F^\ke_g({\bf t}(\psi))=F^\infty_g({\bf t}(\psi)+J_1^\ke(q)),
\end{equation}
or, in the version of Conjecture \ref{equivalent}, to
\begin{equation}\label{numerical}
\begin{split}
&\sum_\beta q^\beta \langle \delta_1\psi_1^{a_1},\dots,\delta_k\psi_k^{a_k}
\rangle^\ke_{g,k,\beta}=\\
& \sum_\beta q^\beta\sum_{m=0}^\infty \frac{1}{m!}\left\lan \delta_1\psi_1^{a_1},\dots,\delta_k\psi_k^{a_k}, {J^\ke_1(q)},\dots, {J^\ke_1(q)}
\right\ran_{g,k+m,\beta}^\infty ,
\end{split}
\end{equation}
for  arbitrary (fixed) $(g, k,\ke)$, arbitrary integers $a_1,\dots a_k\geq 0$, and arbitrary cohomology classes 
$\delta_1,\dots ,\delta_k \in H^*(\WmodG)$. The sums on both sides are over all $\beta$ with $2g-2+k+\ke\beta(L_\theta)>0$.

We upgrade in this section the numerical equality \eqref{numerical} to a stronger statement at the level of virtual classes.

\subsection{Shifted virtual classes} We will use the notations
$$X:=\WmodG,\;\;\;\; X_0:=W/_{\mathrm{aff}}\G,\;\;\;\; \mathfrak{X}:=[W/\G] $$ from now on for the three quotients associated to $(W,\G,\theta)$.

Fix $(g,k,\ke)$. We write $A:=[k+m]\setminus [k]$ for $m=0,1,2,...$ and denote 
$$ev_{A} =(ev_{k+1},\dots , ev_{k+m}):  \overline{M}_{g, k\cup A}(X, \beta )\lra X^{A}$$
the evaluation map.

For each $\theta$-effective class $\beta$, let $$[J_1^\ke]_\beta\in H^{\leq 2}(X,\QQ)$$ denote the coefficient of the $q^\beta$-term of $J_1^\ke(q)$. Recall that 
$J_1^\ke$ has no constant term with respect to $q$, so this coefficient vanishes for $\beta=0$. It also vanishes if $\beta(L_\theta) >1/\ke$ by \eqref{truncation}.

Define a generating series of {\it $\ke$-shifted virtual classes} 
\begin{equation}\label{shifted class}
\sum_{\beta}q^\beta \sum_{|A|=0}^\infty\frac{1}{|A|!} \sum_{\beta_0+\sum_{a\in A}\beta_a=\beta}
[ \overline{M}_{g, k\cup A}(X, \beta_0 )]^{\mathrm{vir}}\cap ev_{A}^*\left( \otimes_{a\in A} [J_1^\ke]_{\beta_a}  \right) .
\end{equation}
As we remarked above, $ [J_1^\ke]_{\beta_a}$ can be nonzero only when $\beta_a\neq 0$. Hence, for each $\beta$, only finitely many terms contribute to the coefficient of $q^\beta$.

Our goal is to compare \eqref{shifted class} with the corresponding generating series 
$$\sum_{\beta}q^\beta [\QmapXe]^{\mathrm{vir}} .$$
We will do so after push-forward via natural maps to a common target.

Let $\iota :X\lra\PP^N$ be the ($\T$-equivariant) 
embedding over the affine quotient induced by the relative polarization $\cO(\theta)$ (to unburden the notation we write simply $\PP^N$ for the
relative projective space $\PP ^N_{X_0}$). It
induces for each stability parameter $\ke '$ a morphism
\begin{equation}\label{embedding}
\iota_{\ke '} : \mathrm{Q}^{\ke '} _{g, k} (X,\beta)   \lra \mathrm{Q}^{\ke '} _{g, k} (\PP ^N, d(\beta) ),
\end{equation}
where the degree $d(\beta)$ is equal to $\beta(L_\theta)$, see \S 3.1 of \cite{CK0}. 

In addition, for $\ke' >\ke''$, there are  ``contraction of rational tails" morphisms
$$c_{\ke'}^{\ke ''}:  \mathrm{Q}^{\ke'}_{g, k} (\PP ^N, d(\beta) )\lra  \mathrm{Q}^{\ke ''}_{g, k} (\PP ^N, d(\beta) ),$$
described in \cite{MOP}, \cite{Toda} (and also recalled in \S3.2.2 of \cite{CK0}). When $\ke '=\infty$ we write $c^{\ke ''}=c_\infty ^{\ke ''}$ and when $\ke ''=0+$ we write $c_{\ke '}=c_{\ke '}^{0+}$.

Let now $(g,k,\ke,\beta)$ be fixed, with $2g-2+k+\ke\beta(L_\theta)>0$. 
For each $A$ and each decomposition $\beta=\beta_0+\sum_{a\in A}\beta_a$ there is a morphism
$$b_{\{\beta_a\}} :  \mathrm{Q}^{\ke}_{g, k\cup A} (\PP ^N, d(\beta_0) )\lra  \mathrm{Q}^{0+}_{g, k} (\PP ^N, d(\beta) ), $$
see \cite{CK0}, \S3.2.3. Informally, the map $b_{\{\beta_a\}}$ replaces each marking  $a\in A$ with a base-point of length $d(\beta_a)$.
If in addition $\beta_a(L_\theta)\leq 1/\ke$ for all $a\in A$, then $b_{\{\beta_a\}}$ factors as 
\begin{equation*}
 \mathrm{Q}^{\ke}_{g, k\cup A} (\PP ^N, d(\beta_0) ) \stackrel{b_{\{\beta_a\}}^\ke } {\lra} \mathrm{Q}^{\ke}_{g, k} (\PP ^N, d(\beta) ) \stackrel{c_\ke}{\lra}  \mathrm{Q}^{0+}_{g, k} (\PP ^N, d(\beta) ).
\end{equation*}
We have the ($\T$-equivariant) composition
\begin{equation}\label{contraction1}
( \coprod_{A,\beta_0,\beta_a} b_{\{\beta_a\}}^\ke\circ c^\ke \circ \iota_\infty ): \coprod_{A,\beta_0,\beta_a}  \overline{M}_{g, k\cup A}(X, \beta_0 ) 
\lra \mathrm{Q}^{\ke}_{g, k} (\PP ^N, d(\beta) ).
\end{equation}

\begin{Thm}\label{ClassThm} Let $X$ be a semi-positive nonsingular quasi-projective
toric variety of dimension $n$, viewed as a GIT quotient $\CC^{n+r}/\!\!/_\theta (\CC^*)^r$ in the standard way, as in Theorem \ref{Main Thm}.
Let $\T\cong (\CC^*)^{n+r}$ be the natural ``big" torus acting on $X$. Then
\begin{equation}\label{Class equation}
\begin{split}
& (\iota_\ke)_* [\QmapXe]^{\mathrm{vir}} =\\
&\sum _{A,\beta_0,\beta_a} \frac{1}{|A|!} 
(b_{\{\beta_a\}}^\ke \circ c^\ke \circ \iota_\infty )_* \left([ \overline{M}_{g, k\cup A}(X, \beta_0 )]^{\mathrm{vir}}\cap ev_{A}^*( \otimes_{a\in A} [J_1^\ke]_{\beta_a}) \right).
\end{split}
\end{equation}
More generally,
\begin{equation}\label{with insertions}
\begin{split}
& (\iota_\ke)_* \left([\QmapXe]^{\mathrm{vir}} \cap\prod_{i=1}^kev_i^*\delta_i\right)=\\
&\sum _{A,\beta_0,\beta_a} \frac{1}{|A|!} 
(b_{\{\beta_a\}}^\ke \circ c^\ke \circ \iota_\infty )_* \left([ \overline{M}_{g, k\cup A}(X, \beta_0 )]^{\mathrm{vir}}\cap ev_{A}^*( \otimes_{a\in A} [J_1^\ke]_{\beta_a})\prod_{i=1}^kev_i^*\delta_i  \right)
\end{split}
\end{equation}
for all $\delta_1,\dots ,\delta_k \in H^*_{\T, \mathrm{loc}}(X,\QQ)$.
\end{Thm}

\begin{Cor}\label{Class level Fano} If $X$ is a nonsingular projective Fano toric variety, then
$$(\iota_\ke)_*  \left( [\QmapXe]^{\mathrm{vir}} \cap\prod_{i=1}^kev_i^*\delta_i\right)= 
(c^\ke \circ \iota_\infty )_*  \left( [ \overline{M}_{g, k}(X, \beta )]^{\mathrm{vir}}\cap\prod_{i=1}^kev_i^*\delta_i\right).
$$
\end{Cor}

Since the $\psi$-classes on both sides are pulled-back from  $\mathrm{Q}^{\ke}_{g, k} (\PP ^N, d(\beta) )$,
Theorem \ref{ClassThm} immediately implies equation
\eqref{numerical}, and hence Theorem \ref{Main Thm}. In fact, Theorem \ref{ClassThm} 
gives the stronger identification of the $\T$-equivariant $\ke$-quasimap CohFT with the $\ke$-shifted $\T$-equivariant
Gromov-Witten CohFT.

Theorem \ref{ClassThm} holds in fact for {\it any} semi-positive GIT presentation $(\CC^{n+r},(\CC^*)^r,\theta)$ of a toric variety, see \S\ref{GIT presentations}. It is instructive to check it directly
for the simplest nontrivial example.

\begin{EG}\label{point target} Consider the GIT triple $(\CC,\CC^*,\theta)$, with $$\theta=1\in \ZZ\cong\chi(\CC^*).$$ The quotient $X$ is a single point. 
A quasimap from a curve $C$ to the quotient stack $[\CC/\CC^*]$ of class is specified by a line bundle
$\cL$
of degree $d$ on $C$ and a global section $0\neq u\in H^0(C,\cL)$, up to a constant nonzero multiple, i.e., by a divisor of degree $d$ on $C$. In particular $d\geq 0$ is the class of the
quasimap. The base-points are the zeroes of the section $u$. Hence the cone
$\mathrm{Eff}(\CC,\CC^*,\theta)$ is canonically identified with $\NN$ and the Novikov ring is the ring of power series $\QQ[[q]]$. 

The Gromov-Witten moduli spaces 
are empty for $d\neq 0$ and are equal to the moduli spaces of stable curves $\Mgk$ for $d=0$. The virtual class is just the fundamental class.

Let us consider now the stability condition $\ke=0+$. For $d\geq 0$ we have $$\mathrm{Q}^{0+}_{g,k}(X,d)\cong\overline{M}_{g,k|d}/S_d.$$
Here the notation for the mixed moduli spaces $\overline{M}_{g,k | d}$ is taken from \S4.1 of \cite{MOP}.
They are a special case of Hassett's moduli spaces of weighted stable curves, \cite{Hasset}, and
parametrize nodal genus $g$ curves with two sets of markings $\{x_1,\dots, x_k\}$ and $\{\hat{x}_1,\dots ,\hat{x}_d\}$, such that the markings in the first set are distinct, 
away from the nodes of $C$ and from the markings $\hat{x}_j$, while the markings in the second set are away from the nodes, 
but $ \hat{x}_j$ and $ \hat{x}_k$ are allowed to coincide. The quotient is by the symmetric group $S_d$, which acts by permuting the second kind of markings.
Again, the virtual class is the fundamental class.

Similarly, we have the identification
$$QG^{0+}_{0,0,d}(X)\cong (\PP^1)^d/S_d\cong \PP^d$$
for the unpointed graph space. From this, the small $I$ function is easily computed to be
$$I_{sm}(q)=e^{q/z},$$
hence $I_1(q)=q$.

From the above discussion, the only splitting $d=d_0+\sum_a d_a$ that contributes to the right-hand side of equation \eqref{Class equation} has $d_0=0$ and $d_a=1$ for all $a\in A$,
so that $A=[d]$. The sum
reduces to the single term
$$\frac{1}{d!} (b_{\{1,\dots ,1\}}\circ c)_*([\overline{M}_{g,k+d}]).
$$
There is a canonical birational map 
$$h:\overline{M}_{g,k+d}\lra \overline{M}_{g,k|d},$$
constructed in \cite{Hasset} and one checks easily that 
$$b_{\{1,\dots ,1\}}\circ c:\overline{M}_{g,k+d} =\overline{M}_{g, k\cup A}(X,0)\lra \mathrm{Q}^{0+}_{g,k}(X,d)=\overline{M}_{g,k|d}/S_d $$ 
coincides with the Hassett contraction $h$ followed by the projection to the quotient by $S_d$.
The equality \eqref{Class equation} now follows.

The above argument generalizes immediately to all stability parameters $\ke$, and to all GIT presentations $(\CC^r,(\CC^*)^r,\theta)$, with $r\geq 2$ and 
$\theta=(1,\dots,1)\in \ZZ^r\cong\chi((\CC^*)^r)$
of the point target $X$. For example, the small I function is $I_{sm}=e^{(q_1+\dots +q_r)/z}$. Note that the number of Novikov parameters changes with the GIT presentation.
\end{EG}

\subsection{The geometric shifting} There is a better, more geometric way to describe the shifted virtual classes. Namely,
for fixed $(g,k,\ke,\beta)$, define the ``mixed" moduli stack  $\overline{M}_{g, k}^{\ke}(X, \beta )$ as 
\begin{equation}\label{M epsilon} 
\overline{M}_{g, k}^{\ke}(X, \beta ) :=\coprod_{|A|=0}^\infty \coprod_{\beta_0+\sum_{a\in A}\beta_a=\beta}
      \overline{M}_{g, k\cup A} (X, \beta _0) \times _{X ^A} \prod _{a\in A} \mathrm{Q}_{0, 1\cup a} ^{\ke } (X, \beta _a),  
      \end{equation}
where the fiber products are taken over the evaluation maps at the markings indexed by $A$. As before, we take the second disjoint union over splittings $\beta_0+\sum_{a\in A}\beta_a=\beta$ with
$\beta_a\neq 0$ for all $a\in A$, so
that only finitely many nonempty fiber products appear in the right-hand side.

We have the cartesian square
\begin{equation*}
\begin{CD}
\overline{M}_{g, k}^{\ke}(X, \beta ) @>\pi_1 >>  \coprod _{A, \beta_0, \beta_a } \overline{M}_{g, k\cup A} (X, \beta _0) \times \prod _{a\in A} \mathrm{Q}_{0, 1\cup a} ^{\ke } (X, \beta _a)\\
@V\pi_2 VV  @VV ev_A^{\overline{M}}\times ev_A^{Q^\ke} V\\
 \coprod _{A}X^A @ > \Delta >>    \coprod _{A}X^A\times X^A
\end{CD}
\end{equation*}
where $\Delta$ is the diagonal map. It endows $\overline{M}_{g, k}^{\ke}(X, \beta )$ with a natural virtual class given by
\begin{equation}
[\overline{M}_{g, k}^{\ke}(X, \beta )]^{\mathrm{vir}} = \Delta^! \sum _{A, \beta_0, \beta_a } \frac{
[\overline{M}_{g, k\cup A} (X, \beta _0)]^{\mathrm{vir}} \otimes (\otimes _{a} [\mathrm{Q}_{0, 1\cup a} ^{\ke } (X, \beta _a)]^{\mathrm{vir}})}{|A|!}.
\end{equation}
The division by $|A|!$ is included to make the markings in $A$ unordered.

Fix $A$ and a splitting $\beta=\beta_0+\sum_{a\in A}\beta_a$. Let
$$p=p_{A,\beta_0,\beta_a}:  \overline{M}_{g, k\cup A} (X, \beta _0) \times _{X ^A} \prod _{a\in A} \mathrm{Q}_{0, 1+a} ^{\ke } (X, \beta _a)\ra  \overline{M}_{g, k\cup A} (X, \beta _0)$$
be the projection.

\begin{Lemma}\label{comparison} We have
\begin{equation*}
\begin{split}
&p_* \Delta^! ([\overline{M}_{g, k\cup A} (X, \beta _0)]^{\mathrm{vir}} \otimes (\otimes _{a\in A} [\mathrm{Q}_{0, 1\cup a} ^{\ke } (X, \beta _a)]^{\mathrm{vir}}))=\\
& [ \overline{M}_{g, k\cup A}(X, \beta_0 )]^{\mathrm{vir}}\cap (ev_{A}^{\overline{M}})^*\left( \otimes_{a\in A} [J_1^\ke]_{\beta_a}  \right) .
\end{split}
\end{equation*}
Hence 
$$\sum_\beta q^\beta  p_*[\overline{M}_{g, k}^{\ke}(X, \beta )]^{\mathrm{vir}} $$
equals the generating series of $\ke$-shifted virtual classes \eqref{shifted class}.
\end{Lemma}
\begin{proof} Consider the fiber product diagram with cartesian squares
\begin{equation*}
\begin{CD}
\overline{M}_A \times_{X^A} \prod _{a\in A} \mathrm{Q}_{a} ^{\ke } @>\pi_1 >>
\overline{M}_A \times \prod _{a\in A} \mathrm{Q}_a^\ke \\
@V p VV  @VV {\mathrm {id}}\times ev_A^{Q^\ke} V\\
\overline{M}_A @> ({\mathrm {id}}, ev_A^{\overline{M}})>> \overline{M}_A\times X^A\\
@V ev_A^{\overline{M}} VV @VV ev_A^{\overline{M}}\times {\mathrm {id}} V\\
X^A @ > \Delta >>   X^A\times X^A ,
\end{CD}
\end{equation*}
where we have used the shorthand notations $\overline{M}_A= \overline{M}_{g, k\cup A} (X, \beta _0)$ and $\mathrm{Q}_a^\ke=\mathrm{Q}_{0, 1+a} ^{\ke } (X, \beta _a)$. The middle
horizontal arrow is the embedding as the graph of the evaluation map $ev_A^{\overline{M}}$, so it is a regular embedding. By standard properties of the refined Gysin maps we have
\begin{equation*}
\begin{split}
p_*\Delta^! ([\overline{M}_A]^{\mathrm{vir}} \otimes &(\otimes _{a\in A} [\mathrm{Q}_a^\ke]^{\mathrm{vir}}))= \\
&\Delta^! ({\mathrm {id}}\times ev_A^{Q^\ke})_*  ([\overline{M}_A]^{\mathrm{vir}} \otimes (\otimes _{a\in A} [\mathrm{Q}_a^\ke]^{\mathrm{vir}}))=\\
& ({\mathrm {id}}, ev_A^{\overline{M}})^!([\overline{M}_A]^{\mathrm{vir}} \otimes (ev_A^{Q^\ke})_*(\otimes _{a\in A} [\mathrm{Q}_a^\ke]^{\mathrm{vir}}))=\\
& [\overline{M}_A]^{\mathrm{vir}} \cap (ev_A^{\overline{M}})^*(\otimes _{a\in A} (ev_a)_*[\mathrm{Q}_a^\ke]^{\mathrm{vir}}).
\end{split}
\end{equation*}
It remains to recall from equation \eqref{mirror map semi} in Proposition \ref{summary} $(ii)$ that
\begin{equation*}
J_1^\ke(q)=\sum_{\beta\neq 0} q^\beta (ev_1)_* [\mathrm{Q}_{0, 2} ^{\ke } (X, \beta )]^{\mathrm{vir}}
\end{equation*}
to conclude the Lemma.
\end{proof}

Let $$s : \overline{M}_{g, k}^{\ke}(X, \beta )\lra \mathrm{Q}^{\ke}_{g, k} (\PP ^N, d(\beta) )$$
denote the composition of the projection $p$ with the morphism \eqref{contraction1}. From  Lemma \ref{comparison} we obtain that Theorem \ref{ClassThm} is equivalent to
\begin{Thm}\label{geometric shifting}
Let $X$ be a toric GIT target as in Theorem \ref{ClassThm}. Then 
$$(\iota_\ke)_* \left([\QmapXe]^{\mathrm{vir}} \cap\prod_{i=1}^kev_i^*\delta_i\right)=s_* \left( [\overline{M}_{g, k}^{\ke}(X, \beta )]^{\mathrm{vir}} \cap\prod_{i=1}^kev_i^*\delta_i\right), $$
for all $(g,k,\beta,\ke)$ with $2g-2+k+\ke\beta(L_\theta)>0$.
\end{Thm} 
      
 There is a parallel construction for graph spaces. Define  the mixed graph space $\overline{M}G_{0, k}^{\ke}(X, \beta )$ to be
 $$
 \coprod_{|A|=0}^\infty \coprod_{\beta_0+\sum_{a\in A}\beta_a=\beta}
      \overline{M}G_{0, k\cup A} (X, \beta _0) \times _{X ^A} \prod _{a\in A} \mathrm{Q}_{0, 1\cup a} ^{\ke } (X, \beta _a),
$$      
with virtual class
\begin{equation*}
\begin{split}
&[\overline{M}G_{0, k}^{\ke}(X, \beta )]^{\mathrm{vir}} =\\
& \Delta^! \sum _{A, \beta_0, \beta_a } \frac{
[\overline{M}G_{0, k\cup A} (X, \beta _0)]^{\mathrm{vir}} \otimes (\otimes _{a} [\mathrm{Q}_{0, 1\cup a} ^{\ke } (X, \beta _a)]^{\mathrm{vir}})}{|A|!}.
\end{split}
\end{equation*}
Here $ \overline{M}G_{0, k\cup A} (X, \beta _0)$ denotes the usual graph space in Gromov-Witten theory.
The analog of Lemma \ref{comparison} holds for these graph spaces, with the same proof.

We also have morphisms 
\begin{equation*}
( \coprod_{A,\beta_0,\beta_a} b_{\{\beta_a\}}^\ke\circ c^\ke \circ \iota_\infty ): \coprod_{A,\beta_0,\beta_a}  \overline{M}G_{0, k\cup A}(X, \beta_0 ) 
\lra QG^{\ke}_{0, k} (\PP ^N, d(\beta) ),
\end{equation*}
$$s=( \coprod_{A,\beta_0,\beta_a} b_{\{\beta_a\}}^\ke\circ c^\ke \circ \iota_\infty \circ p): \overline{M}G_{0, k}^{\ke}(X, \beta ) \ra QG^{\ke}_{0, k} (\PP ^N, d(\beta) ),$$
and
\begin{equation*}
\iota_\ke : QG^\ke_{0,k,\beta}(X) \lra QG^{\ke}_{0, k} (\PP ^N, d(\beta) ).
\end{equation*}

\section{Proof of Theorem \ref{ClassThm}}

\subsection{Overview} 
The idea of the proof, inspired by the work of
Marian, Oprea, and Pandharipande, who treated the case of Grassmannian targets in \cite{MOP}, is to apply $\T$-localization to both the shifted stable map virtual class and to the
quasimap virtual class and then match the push-forward of the localization residues lying over the same $\T$ fixed locus in the space of quasimaps to $\PP^N$.
A genus independence lemma from \cite{MOP} is used to reduce the general case to genus zero. The genus zero toric case requires new ideas, even in the case of Fano targets when no
shifting of virtual classes occurs. We handle it by using $\CC^*$-localization on graph spaces and a localized version of Givental's uniqueness lemma.

As the complete argument is somewhat involved, to keep notation lighter and make the main ideas clear, we present full details in the case $\ke=0+$. The extension to general $\ke$ is
a routine matter, as we indicate in \S\ref{epsilon for toric}. 
The proof is split over several subsections.

\subsection{Toric targets}\label{toric case} Quasimaps to toric targets in their standard GIT presentation were first introduced in \cite{CK} and we very briefly recall the 
description given there to fix some notation. 
Note that the toric varieties were assumed to be projective in \cite{CK} to ensure properness of the moduli spaces, but the description is exactly the same in our more general situation.

Let $X$ be a smooth quasi-projective toric variety of dimension $n$, given by a nonsingular fan $\Sigma$ in $N_\RR$, with $N$ an $n$-dimensional lattice.
Let $\Sigma(1)$ denote the set of rays ($1$-dimensional cones) of $\Sigma$ and put $r:=|\Sigma(1)|-n$. Assume that the rays span $N_\RR\cong\RR^n$ and that $r\geq 1$.
Denote by $\CC ^{\Sigma (1)}$ the vector space spanned by the $1$-dimensional cones and by $\G$ the $r$-dimensional complex torus
$(\CC ^*)^r$. The fan data determines an action of $\G$ on $\CC ^{\Sigma (1)}$ and a ``chamber " in $\chi(\G)$ with the property that for any character in this chamber
we have $X\cong \CC ^{\Sigma (1)}/\!\!/_{\theta}\G$. We fix such a character $\theta$ with $\cO(\theta)$ relatively very ample over the affine quotient.

The coordinates on $\CC ^{\Sigma (1)}$ give ``homogeneous coordinates" 
$(z_\rho)_{\rho\in\Sigma(1)}$ on $X$. 

The ``big" torus ${\T} = (\CC ^*)^{\Sigma (1)}$ acts on  $\CC ^{\Sigma (1)}$ by scaling of coordinates and this
action descends to the quotient stack $\mathfrak{X}=[\CC ^{\Sigma (1)}/\G]$. The induced action on $X$ has isolated $\T$-fixed points, naturally 
corresponding to the maximal cones $\sigma \in \Sigma (n)$. In terms of homogeneous coordinates the fixed
point $p_\sigma$ is described by
$$ z_\rho=1,\;\; {\text{ if}}\; \rho\not\subset\sigma ,\;\;\; z_\rho=0,\;\; {\text {if}}\; \rho\subset\sigma .
$$ 
The $1$-dimensional $\T$-orbits in $X$ are isolated and
correspond to cones of dimension $n-1$ in the fan. We will say that such an orbit is {\it closed} if the corresponding cone is the intersection of
two maximal cones $\sigma_1$ and $\sigma_2$. Such an orbit is a $\PP^1$ joining the two $\T$-fixed points $p_{\sigma_1}$ and $p_{\sigma_2}$.

Let $\xi_\rho$
be the character of the $(\CC^*)^r$-action on the corresponding coordinate axis in $\CC^{\Sigma(1)}$. The associated $\T\times\G$-equivariant line bundle
$$L_{\rho}:=\CC^{\Sigma(1)}\times \CC_{\xi_\rho}$$
descends to the $\T$-equivariant line bundle $\underline{L}_{\rho}$ on $X$. 
For each maximal cone $\sigma$ in the fan, the set $\{ \xi_\rho,\; \rho \not\subset \sigma\}$ is a basis of the character group $\chi(\G)\cong\Pic ([\CC ^{\Sigma (1)}/\G])$.

A quasimap to the quotient stack $[\CC ^{\Sigma (1)}/\G]$
may be described by the data
\begin{equation}\label {toric qmap}
((C, x_1,\dots ,x_k), \{\cL_\rho \}_{\rho\in\Sigma(1)}, \{u_\rho \}_{\rho\in\Sigma(1)})
\end{equation}
where $(C, x_1,\dots ,x_k)$ is a prestable genus $g$ pointed curve, $\cL_\rho,\; \rho\in \Sigma(1)$, are line bundles on $C$, and $u_\rho\in H^0(C,\cL_\rho)$ are global sections,
subject to the nondegeneracy condition
$$(u_\rho(x))_{\rho\in \Sigma(1)}\in (\CC^{\Sigma(1)})^s$$
for all but finitely many points $x\in C$.
For any maximal cone $\sigma \in \Sigma (n)$, the line bundles $\cL_\rho, \; \rho\not\subset\sigma$ uniquely determine all the other. The sections $u_\rho$ are determined up to
the action of $\G$.
If $\beta$ is the class of the quasimap,
we put $$d_\rho:=\deg (\cL_\rho)=\beta(\xi_\rho)\in\ZZ.$$ The set $\{ d_\rho\: | \; \rho\in\Sigma(1)\}$ determines the class $\beta$. 
Since 
$$\det(T_{[\CC ^{\Sigma (1)}/\G]})=\otimes_{\rho\in \Sigma(1)}L_{\rho},$$ 
semi-positivity of the triple $(\CC^{\Sigma(1)},\G,\theta)$ translates 
into the condition 
$$\sum_{\rho\in \Sigma(1)}d_\rho\geq 0$$ 
for all quasimaps to $[\CC ^{\Sigma (1)}/\G]$.

\begin{Rmk}
For this standard GIT presentation we described above, the unstable locus for the linearization $\theta$ 
has codimension at least $2$ and $\Pic (X)= \Pic ([\CC ^{\Sigma (1)}/\G])$ via restriction.
The basis $\{ \xi_\rho,\; \rho \not\subset \sigma\}$ restricts to the $\ZZ$-basis 
$$\{\underline{L}_{\rho}=\mathcal{O}_X(D_\rho),\; \rho \not\subset \sigma\}$$ 
of $\Pic(X)$, with $D_\rho$ the toric divisor in $X$ given by the equation $z_\rho=0$.
Further,
semi-positivity of the triple $(\CC^{\Sigma(1)},\G,\theta)$ is equivalent to the semi-positivity of the anti-canonical class of the toric variety $X$, as $-K_X=\sum_{\rho\in \Sigma(1)} D_\rho$ and 
$\mathrm{Eff}(\CC^{\Sigma(1)},\G,\theta)$ coincides with the semigroup of integral points in the Mori cone of effective curves in $X$.
\end{Rmk}

The notation $\QmapX$ will be used from now on for the moduli space of $(0+)$-stable quasimaps to $X$. A $\CC$-point in $\QmapX$ is specified by the data \eqref{toric qmap},
such that the base-points 
$$\{x\in C\; |\; (u_\rho(x))_{\rho\in \Sigma(1)}\in (\CC^{\Sigma(1)})^{us}\}$$ 
are away from nodes and markings, and satisfying the $(0+)$-stability 
$$\deg(\omega_C(\sum_i x_i)\otimes \cL_\theta^\ke)>0,\; \forall \ke\in \QQ_{>0}.$$
The line bundle $\cL_\theta$ is obtained by writing $\theta$ in the basis $\{ \xi_\rho,\; \rho \not\subset \sigma\}$ (for some maximal cone $\sigma$) and taking the corresponding
tensor product of the $\cL_\rho$'s.

We describe in the next two subsections the $\T$-fixed loci in the moduli spaces $\overline{M}_{g, k}(X, \beta )$ and $\QmapX$
and their contributions to virtual localization formulas, following \cite{Kont}, \cite{GP}, and \cite{MOP}.

\subsection{$\T$-fixed loci}

\subsubsection{$\T$-fixed loci for stable maps}\label{fixed-maps}
Connected components of the $\T$-fixed loci in $\overline{M}_{g, k}(X, \beta )$ are labeled by
decorated graphs $\Gamma = (V, E)$. If $(C, {\bf x}, f)$ is $\T$-fixed,
the corresponding graph $\Gamma$ is obtained as follows. 

Vertices in $v\in V$ correspond to the connected components $C_v$ of $f^{-1}(X^{\T})$. 

Edges $e\in E$ correspond to irreducible components $C_e$ of the domain curve $C$ which are not contracted by $f$. 
These components $C_e$ are rational curves and the restriction $f_e$ of $f$ to each of them is a multiple cover of a $1$-dimensional {\it closed} $\T$-orbit in $X$
of some degree $\delta_e$, ramified only over the two torus fixed points in the orbit.

We decorate each vertex $v\in V$ with the triple $(\sigma_v, g_v, k_v)$, where  $\sigma _v = f(C_v)$,
$g_v$ is the arithmetic genus of $C_v$ (we set $g_v=0$ if the corresponding connected component of $f^{-1}(X^{\T})$ is a single point), 
and $k_v$ is the set of markings carried by $C_v$.
A vertex $v$ is non-degenerate if
$$2 g_v-2 +\mathrm{val}(v)>0,
$$
where $\mathrm{val}(v)$ is the sum of (the cardinality of) $k_v$ and the number of edges incident to $v$.

Each edge $e\in E$ is decorated with the pair $( \mathrm{Orb}_e , \delta_e)$ consisting of the image
$\mathrm{Orb}_e \in \{ \text{closed}\; 1\text{-dim $\T$-orbits in $X$}\}$ of $C_e$ under $f$ and the covering
number $\delta_e \in \NN_{\ge 1}$, so that $\delta_e[\mathrm{Orb}_e] = f_*[C_e]$. 

Note that the resulting graph is connected, without self-edges, and that we have the compatibility conditions
\[ 1- \chi(\Gamma ) + \sum _v g_v = g, \ \coprod _v k_v = \{1, \dots, k\} \] and
\[ \sum_{e\in E} \delta_e[\mathrm{Orb}_e]=\beta.\]

Up to a finite quotient by automorphisms, the component $F_\Gamma$ attached to the decorated graph $\Gamma$ is isomorphic to
 \[ \prod _{v\in V} \overline{M}_{g_v, \mathrm{val}(v) }, \] with the factors corresponding to degenerate vertices treated as points.
 
Conversely, given a decorated graph satisfying the compatibility conditions above, one obtains a $\T$-fixed stable map by taking for each vertex $v$
a stable curve $C_v$ in the corresponding moduli space, for each edge $e$ the covering map $f_e:C_e\lra \PP ^1\cong\mathrm{Orb}_e$ of degree $\delta_e$,
ramified over $0$ and $ \infty$, and gluing along the graph incidences.
 
\subsubsection{$\T$-fixed loci for stable quasimaps}\label{fixed-qmaps}
There is a similar description for the $\T$-fixed locus of $\QmapX$.
The difference is that this time the graphs $\Gamma '$ will have no tail edges, but instead carry
an additional decoration of base-points. Here a tail edge is an edge for which one of its adjacent vertices has valency 1.
The extra decoration is an assignment to each $v\in V$ of a tuple of nonnegative integers 
$$(d_{\rho , v} \in \NN,\; \rho\in\Sigma(1),\; \rho \not\subset \sigma _v).$$
The component $F_{\Gamma '}$ attached to the graph $\Gamma '$ is isomorphic to
\begin{align*}    \prod _{v\in V} \overline{M}_{g_v, \mathrm{val}(v) | \sum _{\rho \not\subset \sigma _v} d_{\rho, v}}
      \end{align*} 
up to a finite quotient. The (Hassett) mixed moduli spaces $\overline{M}_{g,m | d}$ are described in Example \ref{point target}.
The automorphism group for each vertex contains 
the product of symmetric groups $\prod_{\rho \not\subset \sigma _v} S_{d_{\rho, v}}$ which acts by
permuting the second kind of markings.

A vertex of $\Gamma'$ is {\it non-degenerate} if it satisfies the stability condition
$$2g_v-2+\mathrm{val}(v)+\ke(\sum_{\rho \not\subset \sigma _v} d_{\rho , v})>0,\; \text{for every}\; \ke\in\QQ_{>0}.$$
For degenerate vertices, the corresponding factors in $F_{\Gamma '}$ are again treated as points

The quasimap elements in $F_{\Gamma '}$ are constructed as follows.
First, we view $d_{\rho , v}$ also as an index set with $d_{\rho ,v}$ elements such that $d_{\rho, v}$ are mutually disjoint.
For an element in $\prod _{v\in V} \overline{M}_{g_v, \mathrm{val}(v) | \sum _{\rho \not\subset \sigma _v} d_{\rho, v}}$, 
let $C_v$ be the corresponding $m_v$-pointed genus $g_v$ curve $C_v$ with markings $x_i$
and additional base-point markings $\hat{x}_i$, $i\in \cup_\rho d_{\rho , v}$. For $\rho\not\subset \sigma_v$, put
\[ \cL_\rho =  \mathcal{O}_{C_v} (\sum _{i\in d_{\rho , v} } \hat{x}_i) \]      with the canonical section $u_\rho$, whose divisor of zeroes is $\sum _{i\in d_{\rho , v} } \hat{x}_i$.
The remaining $\cL_\rho$'s are determined in terms of $\cL_\rho, \rho \not\subset\sigma _v$,  and 
we set $u_\rho =0$ for $\rho\subset \sigma_v$.
We obtain a stable toric quasimap by gluing the resulting quasimaps $(C_v, x_i \in m_v, \cL_\rho, u_\rho)$, $v\in V$ along graph incidences with the $2$-pointed genus $0$ map data 
$(C_e\cong \PP ^1, 0, \infty; f_e)$ corresponding to the edges $e$. 

Note that for each vertex the quasimap $(C_v, x_i \in m_v, \cL_\rho, u_\rho)$ carries the class $\beta_v$
determined by 
$$\beta_v(L_\rho))=d_{\rho, v},\; \rho \not\subset \sigma _v.$$ Hence the compatibility condition that must be satisfied by the additional decoration is
\[ \sum_{v\in V}\beta_v +\sum_{e\in E} \delta_e[\mathrm{Orb}_e]=\beta.\]

\subsection{Virtual normal bundles}\label{VirNorm} 
\subsubsection{Stable maps} For graphs $\Gamma$ as in \S\ref{fixed-maps}, let $N^{\mathrm{vir}}_{\Gamma}$ denote the virtual normal bundle to the corresponding $\T$-fixed component
 $F_\Gamma\subset \overline{M}_{g, k}(X, \beta )$. 
The multiplicative inverse of its $\T$-equivariant Euler class, lifted to $\prod _{v\in V} \overline{M}_{g_v, \mathrm{val}(v) }$, is obtained by a standard computation. 
In the form given in \cite{MOP}, it is written as a product of contributions from
vertices, edges, and flags $(v,e)$ consisting of a vertex and an incident edge
\begin{align}\label{mapcont}\frac{1}{\re(N^{\mathrm{vir}}_{\Gamma})} =
 \prod _{v}  \mathrm{MapCont}(v) \prod  _{e} \mathrm{MapCont}(e) \prod _{(v, e)} \mathrm{MapCont}(v, e).
\end{align}
The above grouping is made so that the edge and flag contributions (as well as the contributions of degenerate vertices) are pure weights, 
i.e., they are pulled back from $H^*_{\T, loc} (\Spec \CC)$, 
while each non-degenerate vertex contribution is in
$H^*_{\T, loc} (\overline{M}_{g_v, \mathrm{val}(v)})$. Furthermore, for such a vertex
\begin{align}\label{vertexmapcont}\mathrm{MapCont}(v)= \frac{\re(\mathbb{E} ^\vee \ot T_{\sigma _v}X)}{\re (T_{\sigma_v}X)}\frac{1}{\prod _{e} \frac{w(e)}{\delta_e} - \psi _{e} }
\end{align}
where $\mathbb{E}$ is the Hodge bundle, $\re$ denotes the (equivariant) Euler class, and $T_{\sigma_v}X$ is the $\T$-representation 
on the tangent space to $X$ at the fixed point indexed by 
$\sigma_v$. The product in the denominator is over all edges incident to $v$, $w(e)$ denotes the weight of the $\T$-representation on the tangent space 
$T_{\sigma_v}\mathrm{Orb}_e$, and $\psi_e$ is the Chern class of the cotangent line bundle on $\overline{M}_{g_v, \mathrm{val}(v)}$ at the marking corresponding to $e$.

\subsubsection{Stable quasimaps} Let $F_{\Gamma '} \subset \QmapX$ be a component of the $\bT$-fixed locus, corresponding to a graph $\Gamma'$ as in 
\S \ref{fixed-qmaps}. Our goal in this subsection is to provide an expression simlar to \eqref{mapcont}-\eqref{vertexmapcont} for the multiplicative inverse 
$1/\re (N^{vir}_{\Gamma'})$ of the $\T$-equivariant Euler class of the
{\it absolute} virtual normal bundle, lifted to $\prod _{v\in V} \overline{M}_{g_v, \mathrm{val}(v) | \sum _{\rho \not\subset \sigma _v} d_{\rho, v}}$.

For the computation we will make use of the detailed description of
the obstruction theory for quasimaps from \S 5 of \cite{CK}. At a point $(C, x_1,\dots, x_k, \{L_\rho\}, \{u_\rho\})$ of $F_{\Gamma '}$,
the obstruction theory {\it relative to} $\fM _{g, k}$ (governing deformations of the pairs $(L_\rho,u_\rho)$ of line bundles with sections)
is the virtual
$\T$-representation $$H^\bullet (C, Q):=H^0(C,Q)-H^1(C,Q),$$ where
$Q$ is defined by the Euler sequence on $C$
\begin{equation}\label{Euler sequence} 0 \ra \cO _C ^{\oplus (|\Sigma (1)| -\dim X) } \ra \bigoplus _\rho  \cL_\rho \ra Q \ra 0. \end{equation}
The absolute obstruction theory has an additional piece, consisting of the deformations and the automorphisms of the pointed domain curve $(C,x_1,\dots, x_k)$.

By definition, $N^{vir}_{\Gamma'}$ is the {\it moving part} of the absolute obstruction theory. Accordingly, the inverse of its Euler class is the product of two contributions. One factor is obtained from the moving part of the deformation space of $(C,x_1,\dots, x_k)$ (automorphisms of the pointed domain contribute only to the $\bT$-fixed part of the obstruction theory). It has the same expression as in the case of stable maps:
 \begin{equation}\label{domain cont}  \prod _{(v,e)}\frac{1}{(\frac{w(e)}{\delta_e} - \psi _e )}. \end{equation}
 It remains to calculate the other factor, which is 
the inverse of the $\T$-equivariant Euler class of the moving part of $H^\bullet (C, Q)$. For this, note first that on each non-contracted component $C_e$ of $C$, the
Euler sequence \eqref{Euler sequence} is pulled-back via the map $f_e$ from the Euler sequence presenting the tangent bundle of $X$, while on contracted components $C_v$,
the monomorphism in the sequence is the composition
\[\xymatrix{ \bigoplus _{\rho \not\subset \sigma _v} \cO _{C_v}  \ar[rrr]^{\oplus _{\rho \not\subset \sigma} \sum _{i\in d_{\rho , v}} \hat{x}_i} 
             & & & \bigoplus _{\rho\not\subset \sigma _v}  \cL_\rho \ar[rr]^{(\mathrm{id}, 0)} &&  \bigoplus _\rho \cL _\rho  }.\]
             Using the normalization sequence for the domain curve $C$, 
the contribution we seek 
is the product of three factors, namely the moving parts of
\begin{align}\label{1}  \prod _{v\in V}   \frac{\re\left(H^1(C_v, Q)\right)}{\re\left(H^0 (C_v, Q)\right)}  \\
\label{2} \prod _{e\in E}  \frac{\re\left(H^1(\PP ^1, f_{e} ^*TX)\right)}{\re\left(H^0 (\PP ^1, f_{e}^*TX ) \right)}   \\
\label{3} \prod _{ \mathrm{flags} (v,e)}   \re\left(T_{\sigma _v}X\right) .   \end{align}
The moving part in \eqref{2} and the factor \eqref{3} contribute only pure weights, equal to the analogous contributions in the stable map case.

Finally, we analyze $H^\bullet (C_v, Q)$. Each $\cL_\rho$ has a unique expression 
 $$\cL _\rho = \otimes _{\rho ' \not\subset \sigma _v}\cL _{\rho '} ^{ \otimes a_{v, \rho, \rho ' }},$$ 
 where $a_{v, \rho, \rho ' }$ are integers. Then $\cL_\rho=\cO_{C_v}( \hat{\bf{x}}_\rho)$, where
 $ \hat{\bf{x}}_\rho$ is the divisor $ \sum _{\rho ' \not\subset \sigma _v} a_{v, \rho, \rho ' } (\sum _{i\in d_{\rho' , v}} \hat{x}_i ) $ on $C_v$. It follows that
 $$Q|_{C_v}=\left(\oplus_{\rho \subset \sigma _v} \cO_{C_v}( \hat{\bf{x}}_\rho)\right)\oplus \left(\oplus_{\rho \not\subset \sigma _v} \cO_{C_v}( \hat{\bf{x}}_\rho)| _{\hat{\bf{x}}_\rho}\right).$$
 The second term has trivial linearization and its $H^0$ gives the fixed part of $H^\bullet (C_v, Q)$. As for the first term, for each $\rho \subset \sigma _v$ let
 $$\mathrm{Orb}_{\sigma _v, \rho }$$ denote the one dimensional $\T$-orbit corresponding to the $(n-1)$-dimensional cone spanned
 by  $\sigma _v \setminus \{\rho\}$ (this orbit need not be closed),
 so that
 $$T_{\sigma_v}X=\oplus_{\rho \subset \sigma _v}T_{\sigma_v}\mathrm{Orb}_{\sigma _v, \rho }$$
 as $\bT$-representations. The linearization of $\cL_\rho=\cO_{C_v}( \hat{\bf{x}}_\rho)$ is given by 
 the weight of the representation $T_{\sigma_v}\mathrm{Orb}_{\sigma _v, \rho }$. 
Write $\hat{{\bf x}}_\rho=\hat{ {\bf x}}_\rho^+ -\hat{{\bf x}}_\rho^-$, with $\hat{{\bf x}}_\rho^+$ and $\hat{{\bf x}}_\rho^-$ effective divisors. From the equality
 
$$[\cO_{C_v}( \hat{\bf{x}}_\rho)]=[\cO_{C_v}( \hat{\bf{x}}_\rho^+)| _{\hat{\bf{x}}_\rho^+}]-[\cO_{C_v}( \hat{\bf{x}}_\rho^+)| _{\hat{\bf{x}}_\rho^-}]+[\cO_{C_v}]
$$
in the $K$-group of $C_v$,
 we obtain the contribution of the moving part of (\ref{1}) to the inverse Euler class in the form
 \begin{align}\label{half vertex} \frac{\re(\mathbb{E} ^\vee \ot T_{\sigma _v}X)}{\re (T_{\sigma_v}X)}
 \frac{ \prod _{\rho \subset \sigma _v } \re(H^0 (C_v, \cO_{C_v}( \hat{\bf{x}}_\rho^+)|_{\hat{\bf{x}}_\rho^-}) \ot T_{\sigma _v}\mathrm{Orb}_{\sigma _v, \rho})}
 { \prod _{\rho \subset \sigma _v } \re(H^0 (C_v, \cO_{C_v}( \hat{\bf{x}}_\rho^+)|_{\hat{\bf{x}}_\rho^+}) \ot T_{\sigma _v}\mathrm{Orb}_{\sigma _v, \rho}) }.
 \end{align}

Combining all these,  we obtain the desired factorization 
\begin{align}\label{qmapcont} \frac{1}{\re(N^{\mathrm{vir}}_{\Gamma'})} =
 \prod _{v}  \mathrm{QmapCont}(v) \prod  _{e} \mathrm{QmapCont}(e) \prod _{(v, e)} \mathrm{QmapCont}(v, e).   \end{align}
 
 The contributions from edges and flags are pure weights and match the corresponding factors in \eqref{mapcont}:
 \begin{equation}\label{matching} \mathrm{QmapCont}(e)=\mathrm{MapCont}(e),\;\;\;  \mathrm{QmapCont}(v,e)=\mathrm{MapCont}(v,e).
 \end{equation}
 The same is true about the contributions from degenerate vertices.
 
 By \eqref{domain cont} and \eqref{half vertex}, the contribution from a non-degenerate vertex is
\begin{align}  \mathrm{QmapCont}(v) & = \frac{\re(\mathbb{E} ^\vee \ot T_{\sigma _v}X)}{\re (T_{\sigma_v}X)}\frac{1}{\prod _{e} \frac{w(e)}{d_e} - \psi _{e} }
\mathrm{Q}_{v}, \label{local}\end{align}
with
\begin{align}\label{Qv}
  \mathrm{Q}_{v}
& =\frac{\prod _{\rho \subset \sigma _v } \re(H^0 (C_v, \cO_{C_v}( \hat{\bf{x}}_\rho^+)|_{\hat{\bf{x}}_\rho^-}) \ot T_{\sigma _v}\mathrm{Orb}_{\sigma _v, \rho})}
{ \prod _{\rho \subset \sigma _v }\re( H^0 (C_v, \cO_{C_v}( \hat{\bf{x}}_\rho^+)|_{\hat{\bf{x}}_\rho^+}) \ot T_{\sigma _v}\mathrm{Orb}_{\sigma _v, \rho}) }.\end{align}

We conclude this subsection by noting that the identification of the fixed part in $H^\bullet (C_v, Q)$ with
$H^0(C_v, \oplus_{\rho\not\subset\sigma_v} \cO_{C_v}(\sum_{i\in d_{\rho,v}}\hat{x}_i)|_{\sum_{i\in d_{\rho,v}}\hat{x}_i})$
shows that the virtual fundamental class of $F_{\Gamma'}$ induced by the fixed part of the absolute obstruction theory is the fundamental class itself.

\subsection{Localization}\label{localization} Fix $(g, k, \beta)$ with $2g-2+k\geq 0$. Let $\mathrm{Q}_{g,k}(\PP^N, d(\beta))$ be the moduli space of $(0+)$-stable quasimaps 
of degree $d(\beta)$ to $\PP^N_{X_0}$. 
Denote 
$$
 b\circ c \circ \iota_\infty := ( \coprod_{A,\beta_0,\beta_a} b_{\{\beta_a\}}^{0+}\circ c^{0+} \circ \iota_\infty ).
 $$
We describe next the induced maps $(b\circ c\circ \iota_{\infty} )^{\T}$ and $\iota _{0+}^{\T} $ on the $\T$-fixed loci. 

Let $p_{\sigma}\in X^{\T}\subset (\PP ^N)^{\T}$ be a $\T$-fixed point, corresponding to a maximal cone $\sigma$.

The locus $F_{p_\sigma}$ of $\T$-fixed points in $\mathrm{Q}_{g, k} (\PP ^N, d(\beta) )$ which are supported at $p_{\sigma}$
(i.e., those for which the regular map $f_{reg}:C\lra \PP^N$ induced by the quasimap is a constant map to $p_\sigma$) is parametrized by 
the quotient of the mixed moduli space 
$\overline{M}_{g, k | d(\beta)}$ by the action of the symmetric group $S_{d(\beta)}$ permuting the second kind of markings.

Consider first (for all $A$ and all splittings $\beta=\beta_0+\sum_{a\in A}\beta_a$) 
the components $$F_\Gamma\subset\overline{M}_{g, k\cup A}(X, \beta_0 )$$ of the fixed point loci which are mapped into $F_{p_\sigma}$ 
by $b\circ c\circ \iota_{\infty}$.
They correspond to graphs $\Gamma=(V,E)$ of the following type:
\begin{itemize}
\item the vertex set $V=\{v_0\}\cup V'$ contains a distinguished vertex $v_0$ with $\sigma_{v_0}=\sigma$;
\item the distinguished vertex has genus $g_{v_0}=g$, all other vertices have genus $g_v=0$ and $\Gamma$ has no cycles;
\item the markings in $[k]$ are all assigned to the distinguished vertex $v_0$, while there is no restriction on the assignement of the markings in $A$.
\end{itemize}

Let $E_0\subset E$ denote the subset of edges in $\Gamma$ incident to the distinguished vertex $v_0$. Each such edge is the root of a tree $T_e$ such that the graph $\Gamma$ is the join of
all the $T_e$'s at $v_0$. 
Let $A(v_0)$, respectively $A(e)$, 
denote the subsets of markings from $A$ at $v_0$, respectively at the vertices of $T_e$. Each $T_e$
parametrizes  $\T$-fixed 
genus zero stable maps to $X$ with markings $A(e)\cup\bullet$, which send the marking $\bullet$ to $p_\sigma$. A stable map in $F_\Gamma$ is obtained by gluing these to 
a stable curve in
$ \overline{M}_{g,k\cup A(v_0)\cup E_0}$.
We denote by $\beta_e$ the homology class carried by a stable map parametrized by the tree $T_e$ with root $e\in E_0$, so that $\beta_0=\sum_{e\in E_0}\beta_e$.

The map $(b\circ c\circ \iota_{\infty} )^{\T}$ restricted to $F_\Gamma$ contracts each tree $T_e$  into a base-point of length $d(\beta_e)+\sum_{a\in A(e)}d(\beta_a)$, and 
replaces each marking 
$a\in A(v_0)$ by a base-point
of length $d(\beta_a)$. It follows that, up to finite quotients,
 $(b\circ c\circ \iota_{\infty} )^{\T}$ on $F_\Gamma$
equals the composition 
\begin{equation}\label{ciT} g_\Gamma\circ p : \overline{M}_{g,k\cup A(v_0)\cup E_0}\times (\prod_{v\neq v_0}\overline{M}_{0,\mathrm{val}(v)})\lra \overline{M}_{g, k| d(\beta)},
\end{equation}
where
$$p : \overline{M}_{g,k\cup A(v_0)\cup E_0}\times (\prod_{v\neq v_0}\overline{M}_{0,\mathrm{val}(v)})\lra \overline{M}_{g,k\cup A(v_0)\cup E_0}$$
is the projection and
$$g_\Gamma: \overline{M}_{g,k\cup A(v_0)\cup E_0}\lra \overline{M}_{g,k|A(v_0)\cup E_0}\lra \overline{M}_{g, k| d(\beta)}, $$
with the first arrow the Hassett contraction map,
and the second arrow a composition of diagonal maps increasing the multiplicity of markings $e\in E_0$ by $d(\beta_e)+\sum_{a\in A(e)}d(\beta_a)$ and of markings $a\in A(v_0)$ by $d(\beta_a)$.

Next we look at components $F_{\Gamma'}\subset \QmapX$ mapped to $F_{p_\sigma}$ by $\iota _{0+}^{\T} $. Since rational tails are not allowed, 
the graphs $\Gamma'$ have a single vertex $v_0$ and no edges. The vertex is decorated by $\sigma_{v_0}=\sigma$, $g_{v_0}=g$, all markings in $[k]$, and a set of 
integers $\{d_{\rho, v_0}; \rho \not\subset \sigma\}$. The locus $F_{\Gamma'}$ is isomorphic to the quotient 
$$\left(\overline{M}_{g, k | \sum _{\rho \not\subset \sigma }  d_{\rho, v_0}}\right ) / \prod_{\rho \not\subset \sigma }  S_{d_{\rho, v_0}},$$  
where
the product of symmetric groups $\prod_{\rho \not\subset \sigma }  S_{d_{\rho, v_0}}$ acts by permuting
the second kind of markings in the obvious way.

In terms of the basis $\{\xi_\rho\; |\; \rho \not\subset \sigma\}$, the linearization $\theta$ is expressed as
$$\theta=\sum_{\rho \not\subset \sigma} n_{\rho,\sigma}\xi_\rho,$$
with {\it positive} integers $n_{\rho, \sigma}$. 

The restriction of $\iota _{0+}^{\T} $ to $F_{\Gamma'}$ is descended to the corresponding quotients by symmetric groups from the composition of diagonal maps
\begin{equation}\label{iT} \overline{M}_{g, k | \sum _{\rho \not\subset \sigma }  d_{\rho, v_0}}\lra \overline{M}_{g, k| d(\beta)}\end{equation}
which increase the multiplicity of each point in $d_{\rho, v_0}$ by the factor $n_{\rho,\sigma}$. 

For general graphs $\Gamma$ and $\Gamma'$ for which the genus and/or the markings in $[k]$ are distributed among several vertices, 
the map $(b\circ c\circ \iota_{\infty} )^{\T}$ is a product of maps of 
type \eqref{ciT} and $\iota _{0+}^{\T} $ is a product of maps of 
type \eqref{iT}.

We apply the virtual localization formula to the $(0+)$-shifted stable map virtual class:
\begin{equation}\label{stable map residues}
\begin{split} 
&\sum_{A,\beta_0+\beta_a=\beta} \frac{1}{|A|!}[ \overline{M}_{g, k\cup A}(X, \beta_0 ) ]^{\mathrm{vir}}\cap ev_A^* (\otimes_a [I_1]_{\beta_a}) = \\
& \sum_{A,\beta_0+\beta_a=\beta}  \frac{1}{|A|!}i_* \sum _{\Gamma} \frac{[F_\Gamma] \cap i^* ev_A^* (\otimes_a [I_1]_{\beta_a})}{\re(N_\Gamma^{\mathrm{vir}})}.\end{split}
\end{equation}
Similarly, 
\begin{equation}\label{quasimap residues}
[ \QmapX]^{\mathrm{vir}} =i_* \sum _{\Gamma'} \frac{[F_{\Gamma'}]}{\re(N_{\Gamma'}^{\mathrm{vir}})}.
\end{equation}
In both formulas $i$ denotes the inclusion of the $\T$-fixed loci.
In the stable map formula \eqref{stable map residues},
the restrictions  $i^* ev_a^* [I_1]_{\beta_a}$ contribute pure weight factors $\mathrm{cont}(a)$
to the vertices. 

First, we write the vertex contribution as
\begin{align}\label{vertexmapcont1}
\mathrm{MapCont}(v)= \frac{\re(\mathbb{E} ^\vee \ot T_{\sigma _v}X)}{\re (T_{\sigma_v}X)}\frac{1}{\prod _{e\;\text{not collapsed}} \frac{w(e)}{\delta_e} - \psi _{e} } M_v,
\end{align}
where the factor $M_v$ is the product of $1/(\frac{w(e)}{\delta_e} - \psi _{e}) $ over the collapsed edges incident to $v$ and of the contributions from the $A$-markings at $v$.
Next, as in 
\cite {MOP}, for each non-degenerate vertex $v$ 
the factor $M_v$ absorbs the contributions in \eqref{mapcont} coming from all edges and vertices (including their $A$-markings)  of all trees $T_e$ which are collapsed
to $v$ by the map $(b\circ c\circ \iota_{\infty} )^{\T}$. The final form of the vertex contribution is then given by \eqref{vertexmapcont1} with $M_v$ of the form
\begin{equation}\label{Mv}
M_v=\prod _{e\;\text{collapsed to}\; v} \frac{\mathrm{cont}(T_e)}{\frac{w(e)}{\delta_e} - \psi _{e} } \prod_{a\in A(v)} {\mathrm{cont}(a)}.
\end{equation}

We now compare \eqref{vertexmapcont1} with the quasimap vertex contribution
$$\mathrm{QmapCont}(v) = \frac{\re(\mathbb{E} ^\vee \ot T_{\sigma _v}X)}{\re (T_{\sigma_v}X)}\frac{1}{\prod _{e} \frac{w(e)}{d_e} - \psi _{e} }
\mathrm{Q}_{v}$$ from
\eqref{local}.
The first two factors in each formula are pulled back via $(b\circ c\circ \iota_{\infty} )^{\T}$ and $\iota _{0+}^{\T} $ respectively. 
From the projection formula and the matching of the pure weight contributions from 
non-collapsed edges and flags in \eqref{matching}, we conclude that Theorem  \ref{ClassThm} follows from the following Lemma.

\begin{Lemma}\label{vertexmatching} For each $\T$-fixed locus $F_{p_\sigma}\subset \mathrm{Q} _{g, k} (\PP ^N, d(\beta) )^{\T}$ as described in this subsection, the equality 
\begin{equation}\label{eqvertexmatching}\sum_{A,\beta_0+\beta_a=\beta} \frac{1}{|A|!}\sum_{F_{\Gamma} \mapsto F_{p_\sigma}}(b\circ c\circ \iota_{\infty} )_*^{\T}M_{v_0}
=\sum_{F_{\Gamma'} \mapsto F_{p_\sigma}}(\iota _{0+}^{\T})_*\mathrm{Q}_{v_0}
\end{equation}
holds 
in $H^*_{\T , loc} (\overline{M}_{g, k | d(\beta) }/S_{d(\beta)})$. 
\end{Lemma}

We will prove the lemma in two steps. First we use \S7.6 of \cite{MOP} to reduce to a statement in genus zero. 
The genus zero case is then handled by proving a localized version of Givental's Uniqueness Lemma and an inductive argument.

\subsection{MOP Lemma and reduction to genus zero}\label {MOP Lemma}

The mixed moduli spaces $\overline{M}_{g, k | d}$ carry cotangent line classes $\psi_i,\; i=1,\dots ,k$,  and $\hat{\psi}_j,\; j=1,\dots d$.
In addition, there are diagonal classes $$D_J\in H^{2(|J|-1)}(\overline{M}_{g, k | d},\QQ)$$ for $J\subset\{1,\dots, d\}$ with $|J|\neq \emptyset$, corresponding to the 
locus where the markings $\{\hat{x}_j\}_{j\in J}$ coincide.

By the cotangent calculus in \cite{MOP},
each side of \eqref {eqvertexmatching} can be written as a polynomial expression in $\hat{\psi} _j$ and $D_J$ (with coefficients in the field $K=\QQ(\{\lambda_\rho, \rho \in\Sigma(1)\})$, 
depending on $\sigma$ and $\beta$, but {\it independent}
on the genus $g$ and the number $k$ of usual markings. For the left-hand side, this follows from \S 4.3 of \cite{MOP} and the description \eqref{ciT} of $b\circ c\circ \iota_{\infty}$, 
while for the right-hand side we use the formula \eqref{Qv} for $\mathrm{Q}_v$ and \S4.6 in \cite{MOP}, together with \eqref{iT}.

Furthermore, these two polynomials are symmetric in the variables $\hat{\psi} _j$ and may be written in {\it canonical  form} as in \S4.4 of \cite{MOP}. This means that each monomial is rewritten in the form 
$$\hat{\psi}_{J_1}^{s_1}\hat{\psi}_{J_2}^{s_2}\dots \hat{\psi}_{J_l}^{s_l}D_{J_1}D_{J_2}\dots D_{J_l}$$ 
with $J_i$ mutually disjoint and 
$$\hat{\psi}_{J_i}=\hat{\psi}_j|_{D_{J_i}},\;\; \forall j\in J_i.$$
The canonical forms are also symmetric in the $\hat{\psi} _j$'s.
Denote by $P^{\infty}_{\beta , \sigma }$ the canonical form of the left-hand side of \eqref {eqvertexmatching}, and by $P^{0+}_{\beta , \sigma }$  the canonical form of
the right-hand side of \eqref {eqvertexmatching}.
We will show that $P^\ke_{\beta _v, \sigma _v}$
as an {\em abstract polynomial} does not depend on $\ke\in\{0+,\infty\}$.

Fix $k\geq 3$, $d\geq 0$, and $1\leq \ell\leq k-2$.
Let $$\cP=(\cP _1, ..., \cP _\ell) $$ be a set partition of $\{1,..., d\}$ with $|\cP _i|\ge 1$.
Let $\tau:=(t _1, ..., t_\ell)$ be an ordered partition of $k-2-\ell$ (the integers $t_i$ are nonnegative, but may be zero).  Following [MOP, Lemma 6],
we associate to the above data a chain-type topological stratum $S(\tau, \cP)$
on $\oM _{0, k | d}$. When $\ell\geq 2$, the generic element in the stratum has $\ell$ irreducible components,
$R_1,\dots, R_\ell$, attached tail to head in a chain of rational curves, with the markings distributed as follows:
\begin{itemize}
\item $R_1$ carries $t_1+2$ markings and the base-point markings in $\cP_1$.

\item $R_i$ carries $t_i+1$ markings and the base-point markings in $\cP_i$, for $i=2, ...,\ell-1$.

\item $R_\ell$ carries $t_\ell +2$ markings the base-point markings in $\cP_\ell$ . 
\end{itemize}
The usual markings $x_1,\dots x_k$ are distributed in order from left to right.
When $\ell=1$, the stratum is simply the entire space $\oM _{0, k | d}$.

\begin{Lemma}\label{MOPlemma} {\em [MOP, \S 7.6]} Fix an integer $d>0$. Consider formal variables $\hat{\psi} _{j}$ for $j=1,\dots,d$ and
$D_J$ for nonempty $J\subset \{1, ..., d\}$.
Let $\Delta$  be a polynomial in  $\hat{\psi} _{j}$, $D_J$, in canonical form. For every $k\ge 3$, we may view $\Delta$ as a class in $H^*_{\T, \mathrm{loc}}(\overline{M}_{0, k| d})$.
Then $\Delta=0$ as an abstract polynomial if and only if 
$$\int _{S(\tau, \cP)} \mu (\psi _1, ..., \psi _k ) \Delta =0$$ 
for every topological stratum $S(\tau, \cP)$ as above and every monomial $\mu$ in $\psi _i$.
Further, if $\Delta$ is symmetric in the $\hat{\psi} _{j}$'s, then only vanishing of the integral on $\overline{M}_{0, k| d}/S_d$ is required to conclude that $\Delta$ vanishes as an abstract polynomial.
\end{Lemma}

\begin{proof}
This is essentially Lemma 6 in \cite{MOP}. Their statement is formulated to require the vanishing of the integrals over all possible topological strata in $\overline{M}_{0, k| d}$, however, the proof they give shows that it suffices to consider only strata of the form $S(\tau, \cP)$. (Moreover, one may also restrict to a very special kind of monomials 
$ \mu (\psi _1, ..., \psi _k )$, but we will not need this fact.)
\end{proof}

Applying Lemma \ref{MOPlemma} to $\Delta=P^\infty _{\beta , \sigma }-P^{0+} _{\beta , \sigma }$, the proof of Lemma \ref{vertexmatching} is reduced to proving the following Lemma.

\begin{Lemma}\label{independence} For every $k\ge 3$, every chain-type topological stratum $S(\tau, \cP)\subset \overline{M}_{0, k| d(\beta)}$, and every 
polynomial $$\mu(\psi)=\mu(\psi _1, ..., \psi _k)$$ with coefficients in the
field $K=H^*_{\T, loc}(\Spec \CC)$, the genus zero intersection number
\begin{equation}\label{ind}  P_{\beta , \sigma }^{ \ke}(k,\tau, \cP, \mu )
:= \int _ {S(\tau, \cP)} P^\ke _{\beta , \sigma }(\hat{\psi} _j , D_J) \mu (\psi )   \end{equation}
does not depend on $\ke$.
\end{Lemma}

\subsection{Uniqueness Lemma}\label{uniqueness}

Up to this point, the argument for proving Theorem \ref{ClassThm} has been entirely parallel to the one given in \cite{MOP} for the analogous
statement in the case of Grassmannian targets. In their situation no shifting of virtual classes is needed, 
and the $\ke$-independence of \eqref{ind} is an immediate consequence of the fact that in genus zero the 
moduli spaces of stable maps and stable quasimaps to the Grassmannian are smooth and irreducible, of the expected dimension. 

The latter property fails for our toric targets. Even for the Fano cases, which do not require the shifting of virtual classes,
a new idea is needed to complete the proof.
To this end we will use graph spaces and localization with respect to their additional $\CC^*$-action to obtain a localized version of Givental's uniqueness lemma.

For $k,d\geq 0$ denote by $\PP^1 [ k | d ]$ the moduli space parameterizing stable genus 0 curves with a rigid $\PP ^1$ component,
$k$ usual markings, and $d$ ordered base-point markings. By matching the stability conditions we get 
$$\PP^1 [ k | d ]/S_d=QG_{0, k, d} (\PP ^0),$$
the quasimap graph space with target $\PP^0=\CC/\!\!/\CC^*$ (and stability parameter $\ke=0+$).

The $\T$-fixed locus of 
$QG_{0, k, d(\beta)} (\PP ^N)$ whose elements are supported only on $p_\sigma \in X^{\T}\subset (\PP ^N)^{\T}$ may be viewed as $QG_{0, k, d(\beta)} (p_\sigma=\CC/\!\!/\CC^*)$,
and is therefore isomorphic 
to $\PP^1 [ k | d(\beta) ]/ S_{d(\beta)}$. 
We will denote it by $\PP^1 [ k | d(\beta) ] _\sigma$.

Let $\overline{M}G_{0, k, \beta}(X)$ denote the usual stable map graph spaces.
Consider the $(0+)$-shifted virtual class
$$[\overline{M}G_{0, k, \beta}^{\mathrm{shifted}}(X)]^{\mathrm{vir}}:=
\sum_{A,\beta_0+\beta_a=\beta} \frac{1}{|A|!} [\overline{M}G_{0, k\cup A, \beta_0}(X)]^{\mathrm{vir}}\cap ev_A^* (\otimes_a [I_1]_{\beta_a}) $$
and take its $\T$-localization residue 
supported only at $p_\sigma$ under the contraction map $b\circ c\circ \iota _{\infty}$ on graph spaces. It is given by
\begin{align}\label{Res} 
&\mathrm{Res}_\sigma ([\overline{M}G_{0, k, \beta}^{\mathrm{shifted}}(X)]^{\mathrm{vir}})  =\\ 
\nonumber &
(b\circ c\circ \iota _\infty)^{\T}_*  \left( \sum_{A,\beta_0+\beta_a=\beta} \frac{1}{|A|!}
\sum _\Gamma \frac{ [G_\Gamma] } {\mathrm{e}(N_{G_\Gamma}^{\mathrm{vir}} )} \cap (ev_A^* (\otimes_a [I_1]_{\beta_a}))|_{G_\Gamma}\right), \end{align} 
the inner sum over all $\T$-fixed components $G_\Gamma$ in $G_{0, k\cup A, \beta_0}(X)$
which are mapped to $\PP^1 [ k | d(\beta) ] _\sigma$ by $b\circ c\circ \iota _\infty$.

The residue \eqref{Res} is an element of the localized $\T\times \CC^*$-equivariant homology group  $H_*^{\CC ^*\times \T ,\T\text{-loc}} (\PP^1 [ k | d(\beta) ])$, where
{\it only the $\T$-parameters} are inverted. 

The $\T$-fixed loci $G_\Gamma$ correspond to decorated 
graphs $\Gamma=(V,E)$ with a distinguished vertex $v_0$, labelled by $\sigma$ and carrying all markings in $[k]$, as in \S\ref{localization}. 
Up to a finite quotient by automorphisms, the component $G_\Gamma$ is isomorphic to the product
$$\PP^1[\mathrm{val}(v_0)]\times\prod_{v\neq v_0}\overline{M}_{0,\mathrm{val}(v)}.$$
Here $\PP^1[\mathrm{val}(v_0)]$ is the Fulton-MacPherson moduli space 
of stable genus zero marked curves with a rigid component.

Consider the action by $\CC^*$ on a component $G_\Gamma$. The fixed points are obtained by taking two stable maps to $X$ which are supported at $p_\sigma$ under
$b\circ c\circ \iota _{\infty}$, each with one extra marking, and attaching them to the rigid $\PP^1$ at $0$ and $\infty$ using the respective additional markings.

It follows that the components of the $\CC^*$-fixed loci in $G_\Gamma$
are isomorphic to products 
$F_{\Gamma _1} \times F_{\Gamma _2},$
where 
\begin{itemize}
\item each $F_{\Gamma _i}$ is a $\T$-fixed component in $\overline{M}_{0, B_i\cup A_i\cup \bullet}(X,\beta_0^i)$ supported at $p_\sigma$ under $b\circ c\circ\iota_{\infty}$ ,
\item 
$B_1\coprod  B_2 = \{1, ..., k\}$, $A_1\coprod A_2=A$ and $ \beta _0^1+ \beta _0^2 = \beta_0$, 
\item the graphs $\Gamma_1$ and $\Gamma_2$ satisfy $\Gamma _1 \star \Gamma _2 = \Gamma$,
where the operation $\star$ means joining at the two distinguished non-degenerate vertices and deleting the two additional markings from the decoration of the resulting graph.
\end{itemize}

For fixed $B_1\coprod  B_2 = \{1, ..., k\}$ and $\beta_1+\beta_2=\beta$ we collect together all components $F_{\Gamma _1} \times F_{\Gamma _2}$ with the given splitting of the markings in $[k]$
and with 
$$\beta_0^1+\sum_{a\in A_1}\beta_a=\beta_1,\;\;\;  \beta_0^2+\sum_{a\in A_2}\beta_a=\beta_2.$$
The map $b\circ c\circ\iota_\infty$ on the $(\CC^*\times\T)$-fixed locus  
$$\coprod _{\stackrel{B_1\coprod  B_2 =[k]} {\beta _1+ \beta _2 = \beta}}F_{\Gamma _1} \times F_{\Gamma _2}$$ in the disjoint union of graph spaces is the composition
the product of contraction maps  
on moduli of unparametrized stable maps with the inclusion
\begin{equation}\label{gluing} 
\overline{M}_{0, B_1\cup\bullet| d(\beta _1)}/S_{d(\beta_1)}\times
\overline{M}_{0, B_2\cup\bullet| d(\beta _2)}/S_{d(\beta_2)}\hookrightarrow \PP^1 [k | d(\beta )]_\sigma.
\end{equation}

In the stable cases $B_1, B_2\neq\emptyset$, the inclusion is obtained by attaching two stable quasimaps to $p_\sigma\in \PP^N$ at $0$ and $\infty$ on the rigid $\PP ^1$.
If for example $B_1=\emptyset$, the inclusion is obtained by taking a degree $d(\beta_1)$ quasimap from the rigid $\PP^1$ to $p_\sigma$ which has at $0$ a base-point
of multiplicity $d(\beta_1)$, and gluing to it at $\infty$ a stable quasimap to $p_\sigma$ from $\overline{M}_{0, k\cup \bullet| d(\beta _2)}$.

Applying $\CC ^*$-localization and summing over all $(A,\beta_0, \beta_a)$ gives a factorization expression 
\begin{align}\label{factorization maps} & \mathrm{Res}_\sigma ([\overline{M}G_{0, k, \beta}^{\mathrm{shifted}}(X)]^{\mathrm{vir}})= \\
\nonumber & = \sum _{\stackrel{B_1\coprod  B_2 =[k]} {\beta _1+ \beta _2 = \beta}} \frac{1}{[T_{p_\sigma}X]}
  \left(\frac{P^\infty_{\sigma, \beta _1} [\overline{M}_{0, B_1\cup\bullet| d(\beta _1)}] }{z (z-\psi _\bullet)} \right) \star
\left(\frac{P^{\infty}_{\sigma, \beta _2} [\overline{M}_{0, B_2\cup\bullet| d(\beta _2)}]}{-z (-z-\psi _\bullet)}\right), 
\end{align}
where $\star$ means the operation 
\[ \star : H_*( \overline{M}_{0, B_1\cup\bullet| d(\beta _1)}  ) \otimes H_* ( \overline{M}_{0, B_2\cup\bullet| d(\beta _2)}) \ra H_* (\PP^1 [k | d(\beta )]) \]
induced by the inclusion \eqref{gluing}.

The notation in the formula \eqref{factorization maps} requires more explanation.
We write $P^\infty_{\sigma, \beta _i} [\overline{M}_{0, B_i\cup\bullet| d(\beta _i)}]$ to indicate that $P^\infty_{\sigma, \beta _i} $, 
which is independent of the number $k$ of usual markings, is evaluated in $H^*_{\T , loc} (\overline{M}_{0,  B_i\cup\bullet| d(\beta_i) })$ by taking the cap product with the fundamental class.

The product of denominators $z(z-\psi_\bullet)(-z)(-z-\psi_\bullet)$ is the (well-known) $\CC^*$-equivariant Euler class of the normal bundle to 
$$\overline{M}_{0, B_1\cup\bullet| d(\beta _0^1)}\times
\overline{M}_{0, B_2\cup\bullet| d(\beta _0^2)}$$ in $\PP^1 [k | d(\beta )]$, with $\psi_\bullet$  the cotangent line classes at the additional markings of
$\overline{M}_{0, B_i\cup\bullet| d(\beta _i)}$, and $z$ the equivariant parameter. 

The formula \eqref{factorization maps} is correct as written for the stable cases $B_1,B_2\neq\emptyset$, but for the unstable cases the notation is abused 
and should be understood as the following convention: 
\begin{equation}\label{empty-map} \begin{split} &\left(\frac{P^\infty_{\sigma, \beta _1} [\overline{M}_{0, \bullet| d(\beta _1)}] }{z (z-\psi _0)} \right) := i^*_\sigma [J(q,I_1(q),z)]_{\beta_1},\\
&\left(\frac{P^\infty_{\sigma, \beta _2} [\overline{M}_{0, \bullet| d(\beta _2)}] }{-z (-z-\psi _\infty)} \right) := i^*_\sigma [J(q,I_1(q),-z)]_{\beta_2},
\end{split}\end{equation}
where $i_\sigma:\{ p_\sigma\} \ra X$ is the inclusion and $[J(q,I_1(q),z)]_{\beta}$ is the coefficient of $q^{\beta}$ in the {\it mirror map-transformed} small $J$-function of $X$. Precisely,
\begin{equation}\label{mirror small J}
\begin{split}
&J(q,I_1,z)=\one+\frac{I_1(q)}{z}+\\
&\sum_{(\beta' ,m)\neq(0,1)} \frac{q^{\beta'}}{m!}(ev_\bullet)_*\frac{[\overline{M}_{0,m\cup \bullet}(X,\beta')]^{\mathrm{vir}}\cap \prod_{j=1}^{m}ev_j^*(I_1(q))}{z (z-\psi_\bullet )}.
\end{split}
\end{equation}
Note that by this definition $i^*_\sigma [J(q,I_1(q),z)]_{\beta}$ is an element of $$H^*_{\T, \mathrm{loc}}(\Spec\CC)[[1/z]]=\QQ(\{\lambda_\rho\})[[1/z]].$$
However, its appearance as a localization contribution in the factorization 
\eqref{factorization maps} for $k=0$ shows that the $1/z$-series can be summed to a rational function in 
$H^*_{\CC^*\times\T, \mathrm{loc}}(\Spec\CC)=\QQ(\{\lambda_\rho\}, z)$.

The same argument for the quasimap graph space $QG_{0, k, \beta}(X)$ produces the factorization

\begin{align} 
 \mathrm{Res}_\sigma &([QG_{0, k, \beta}(X)]^{\mathrm{vir}}) = \\
\nonumber &=(\iota _{0+})_*
 \sum _{\Gamma' : \mathrm{Supp}(\Gamma' )=p_\sigma} \frac{ [G_{\Gamma'}] } {\mathrm{e}(N_{G_{\Gamma'}}^{\mathrm{vir}} )}=\\
\nonumber & =\sum _{ \stackrel{B_1\coprod  B_2 = [k]} {\beta _1+ \beta _2 = \beta}} \frac{1}{[T_{p_\sigma}X]}
 \left(\frac{P^{0+}_{\sigma, \beta _1} [\overline{M}_{0, B_1\cup\bullet| d(\beta _1)}] }{z (z-\psi _\bullet)} \right) \star
\left(\frac{P^{0+}_{\sigma, \beta _2} [\overline{M}_{0, B_2\cup\bullet| d(\beta _2)}]}{-z (-z-\psi _\bullet)}\right).
\end{align}
Again a convention is used in the unstable cases:
\begin{align}\label{empty-quasi} 
&\frac{P^{0+}_{\sigma, \beta _1}[\overline{M}_{0, \bullet| d(\beta _1)}] }{z (z-\psi _\bullet)}  := i^*_\sigma I_{\beta_1 }(z),\\
\nonumber & \frac{P^{0+}_{\sigma, \beta _2}[\overline{M}_{0, \bullet| d(\beta _2)}] }{-z (-z-\psi _\bullet)}  := i^*_\sigma I_{\beta_2 }(-z),
\end{align}
with $I_{\beta}(z)$ the degree $\beta$ part of the small $I$-function of the toric variety $X$. It is given explicitly (see \cite{Givental}, or \S7 of \cite{CK}) by the formula
$$i^*_\sigma I_{\beta }(z)=\prod_{\rho\in\Sigma(1)}\frac{\prod_{j=-\infty}^0(i^*_\sigma c_1^{\T}(\underline{L}_\rho)+jz)}{\prod_{j=-\infty}^{d_\rho}(i^*_\sigma c_1^{\T}(\underline{L}_\rho)+jz)}.
$$
We may view $i^*_\sigma I_{\beta }$ either as an element of $\QQ(\{\lambda_\rho\},z)$ or, by expanding the geometric series, as an element in $\QQ(\{\lambda_\rho\})[[1/z]]$.

Given a monomial $\mu_B(\psi)=\prod _{i\in B}  \psi _i^{\alpha_i}$ and $\ke\in\{0+,\infty\}$, we define
\begin{equation}\label{descendant}
\langle \frac{1}{z (z-\psi _\bullet)}, \mu_B(\psi) \rangle _{0, B\cup\bullet, \beta}^{\ke, p_\sigma}:=
\int _{\overline{M}_{0, B\cup\bullet| d(\beta)}}\frac{P^{\ke}_{\sigma, \beta _1}}{z (z-\psi _\bullet)}\mu_B(\psi).
\end{equation}
For the unstable cases $B=\emptyset$, we use the same convention as before:
$$\langle \frac{1}{z (z-\psi _\bullet)} \rangle _{0, \bullet, \beta}^{\infty, p_\sigma}:=i_\sigma ^*[J (q,I_1(q),z)]_\beta, \;\;\;\; 
\langle \frac{1}{z (z-\psi _\bullet)} \rangle _{0, \bullet, \beta}^{0+, p_\sigma}:=i_\sigma ^*I_\beta(z).
$$
In all cases, it is an element in the field 
$$H^*_{\CC^*\times\T, loc}(\Spec \CC)=\QQ(\{\lambda_\rho\}, z).$$
The notation with superscript $\sigma$ is chosen to reflect that \eqref{descendant} for $\ke=0+$
is the localization
contribution from $\T$-fixed loci supported only over $p_\sigma$ to the degree $\beta$ part in the generating series of descendant invariants 
\begin{equation}\label{series 0+}
\lla \frac{1}{z (z-\psi _\bullet)}, \mu_B(\psi) \rra _{0, B\cup\bullet}^{0+}(t)|_{t=0},
\end{equation}
and the same is true when $\ke=\infty$ for the series
\begin{equation}\label{series infty}
\lla \frac{1}{z (z-\psi _\bullet)}, \mu_B(\psi) \rra _{0, B\cup\bullet}^{\infty} (t+I_1(q))|_{t=0}.
\end{equation}
Note that the two series are equal by Theorem \ref{genus zero intro}, but for $B\neq\emptyset$ this {\it does not}
directly imply the sharper equality of the local contributions \eqref{descendant} at each $p_\sigma$.

However, for $B=\emptyset$ the series \eqref{series 0+} is the small $I$-function of $X$,
\begin{equation*}\label{toric small I}
I^X_{sm}(q,z)=\one +\sum_{\beta\neq 0} q^\beta \prod_{\rho\in\Sigma(1)}\frac{\prod_{j=-\infty}^0 (c_1^{\T}(\underline{L}_\rho)+jz)}{\prod_{j=-\infty}^{d_\rho}(c_1^{\T}(\underline{L}_\rho)+jz)},
\end{equation*}
while the series \eqref{series infty} is the mirror-map transform $J(q,I_1,z)$ of the small $J$-function of $X$, see \eqref{mirror small J}.
Their equality (which is of course the celebrated Givental's toric Mirror Theorem, \cite{Givental}) does give the required local equality
\begin{equation}\label{I=J}
i_\sigma ^*[J (q,I_1(q),z)]_\beta=i_\sigma ^*I_\beta
\end{equation}
for all fixed points $p_\sigma \in X^{\T}$ and all $\theta$-effective $\beta$. 

The following Lemma proves in particular \eqref{ind} of Lemma \ref{independence} for the largest strata $S(\tau,\cP)=\overline{M}_{0, B\cup\bullet| d(\beta)}$.
It is a variant of Givental's Uniqueness Lemma (Proposition 4.5 in \cite{Givental}, or Lemma 3 in \cite{Kim}).

\begin{Lemma} {\em (Localized Uniqueness Lemma)}\label{uniqueness lemma}
For every $B$ (possibly empty), every monomial $\mu_B $ in $\psi$-classes, and every $\theta$-effective class $\beta$, the localized intersection number  \eqref{descendant}
is independent of $\ke\in \{0+,\infty\}$.
\end{Lemma}
 
\begin{proof}  We prove the Lemma by induction on $\beta$ and $k=|B|$. We first observe that:

$(a)$  For $B=\emptyset$, the lemma reduces to \eqref{I=J}, hence it is true by Givental's Theorem.

$(b)$ For any $B$ and $\beta =0$, the lemma is true. Indeed, the terms with $|B|=1$ vanish on both sides, while for $|B|\geq 2$ the moduli spaces coincide with
$\overline{M}_{0,B\cup\bullet}\times X$ and are therefore $\ke$-independent. 

Suppose that the lemma holds true for $\beta ' < \beta$ and for  $k' < k$. Let $k\ge 1$ and $\beta \geq 0$. 

Let $y$ be a formal variable with relation $y^2=0$.
For a polynomial $\CC^*$-equivariant cohomology class $\mu (\psi)= \psi _1^{\alpha_1}...\psi_k ^{\alpha_k}$ on $\PP ^1[k|d(\beta)]$, let
\begin{align} 
D^{\mu, \infty}_{k, \beta, \sigma} &:= \int _{\mathrm{Res}_\sigma ([\overline{M}G_{0, k, \beta}^{\mathrm{shifted}}(X)]^{\mathrm{vir}}) } \mu (\psi ) e^{c_1(U) y } ,
\end{align} 
where $U= U(L_\theta)$ is the universal $(\T\times\CC^*)$-equivariant line bundle described in \S 3.3 of \cite{CK0} and $c_1$ is the equivariant first Chern class. Similarly, put
\begin{align} 
D^{\mu, 0+}_{k, \beta, \sigma} &:= \int _{\mathrm{Res}_\sigma ([QG_{0, k, \beta}(X)]^{\mathrm{vir}}) } \mu (\psi ) e^{c_1(U) y }. 
\end{align} 
Since they are defined without localization with respect to the $\CC^*$-action, the quantities $D^{\mu, \ke}_{m, \beta, \sigma}$ have no pole in $z$, i.e., we have
   $$D^{\mu, \ke}_{m, \beta, \sigma}     \in H^*_{\T, \mathrm{loc}}[y]/(y^2) [[z]].$$

By the factorized expressions we obtain 
\begin{align*} D^{\mu,\ke}_{k, \beta, \sigma}  
&=  \sum _{\stackrel{B_1\coprod  B_2 = [k] } {\beta _1+ \beta _2 = \beta}} \frac{e ^{(w(\cO(\theta)_{p_\sigma}) - d(\beta _2)  z )y }}{[T_{p_\sigma}X]} \times\\
& \langle \frac{1}{z (z-\psi _\bullet)}, \mu _{B_1}\rangle _{0, B_1\cup\bullet, \beta _1}^{p_\sigma,\ke} \langle \frac{1}{z (z-\psi _\bullet)}, \mu _{B_2}\rangle _{0, B_2\cup\bullet, \beta _2}^{p_\sigma,\ke},
\end{align*}
where $w(\cO(\theta)_{p_\sigma})$ is the weight of the $\T$-representation on the fiber of $\cO(\theta)$ at $p_\sigma$. 

Consider the difference  
\begin{align*} \Delta(D _{k, \beta , \sigma}^{\mu,\ke}) :=  D _{k, \beta , p}^{\mu, \infty} - D_{k, \beta, p}^{\mu, 0+} .\end{align*}
By the induction hypothesis we get
\begin{align*}\Delta(D _{k, \beta , \sigma}^{\mu, \ke}) & =  \frac{e ^{w(\cO(\theta)_{p_\sigma})y}}{[T_{p_\sigma}X]} 
\left(\Delta (\langle \frac{1}{z (z-\psi _\bullet)}, \mu \rangle _{0, k\cup\bullet, \beta}^{\ke, p_\sigma} )\right.\\
&\left.+ e^{-d(\beta)  z y} \Delta (\langle \frac{1}{-z (-z-\psi _\bullet)}, \mu \rangle _{0, k\cup\bullet, \beta}^{\ke, p_\sigma} ) \right),  \end{align*}
where 
\begin{align*}
&\Delta (\langle \frac{1}{\pm z (\pm z-\psi _\bullet)}, \mu \rangle _{0, k\cup\bullet, \beta}^{\ke, p_\sigma}):=\\
&\langle \frac{1}{\pm z (\pm z-\psi _\bullet)}, \mu \rangle _{0, k\cup\bullet, \beta}^{\infty, p_\sigma}-
\langle \frac{1}{\pm z (\pm z-\psi _\bullet)}, \mu \rangle _{0, k\cup\bullet, \beta}^{0+, p_\sigma}.
\end{align*}

Since the $\psi$-classes are nilpotent in $H^*_{\CC^*\times\T, \mathrm{loc}}(\overline{M}_{0,k+1|d(\beta)})$, we conclude that
$\langle\frac{1}{z (z-\psi _\bullet)}, \mu \rangle _{0, k+1, \beta}^{\ke, p_\sigma}$ for $\ke\in\{0+,\infty\}$
are polynomials in $1/z$, divisible by $(1/z)^2$. Hence we may write
$$\Delta (\langle \frac{1}{z (z-\psi _\bullet)}, \mu \rangle _{0, k+1, \beta}^{\ke, p_\sigma} )=z^{-2a}(C_1\frac{1}{z}+C_2),$$
with $a\geq 1$.

On the other hand, we have observed that $D _{k, \beta , p}^{\mu, \ke}$ and therefore
$$\Delta(D _{k, \beta , \sigma}^{\mu, \ke})=\frac{e ^{w(\cO(\theta)_{p_\sigma})y}}{[T_{p_\sigma}X]} z^{-2a}\left(C_1\frac{1}{z}+C_2+(1-d(\beta)  z y)(-C_1\frac{1}{z}+C_2)
\right)$$
has no pole in $z$.
This immediately implies $C_2=C_1=0$ and concludes the proof. \end{proof}

\subsection{Conclusion of the proof of Theorem \ref{ClassThm}} In this section we prove Lemma \ref{independence} for all strata $S(\tau,\cP)$. As explained already, this implies that
Lemma \ref{vertexmatching} is true and finishes the proof of Theorem \ref{ClassThm}.
The argument will use induction and is based on a splitting property enjoyed by $P_{\beta , \sigma }^{ \ke}(k,\tau, \cP, \mu )$ which we discuss next.

Given a stratum $S(\tau,\cP)\subset \overline{M}_{0,k|d(\beta)}$ as in \S\ref{MOP Lemma}, we say that $\cP$ is {\it compatible with} $\beta$ if there exists a splitting
$\beta=\beta_1+\dots +\beta_\ell$ with nonzero $\theta$-effective $\beta_i$, and such that $|\cP_i|=d(\beta_i)$. 

For such a compatible stratum and a fixed $1\leq m\leq \ell-1$,
we split the stratum at the $m^{\mathrm{th}}$ node. Precisely, set
$$\tau'=(t_1,\dots,t_m),\;\;\; \tau''=(t_{m+1},\dots t_\ell));$$
$$\cP'=(\cP_1,\dots,\cP_m)),\;\;\; \cP''=(\cP_{m+1},\dots, \cP_\ell);$$
$$k'+1=t_1+\dots +t_m+m+2,\;\;\; k''+1=t_{m+1}+\dots +t_\ell+\ell-m+2 ;$$
$$d'=|\cP_1|+\dots |\cP_m|,\;\;\; d''=|\cP_{m+1}|+\dots |\cP_\ell |.$$
By the compatibility assumption, the (finite) subset 
$$R(\cP,\beta):=\{(\beta',\beta'')\; |\; \beta'+\beta''=\beta, d(\beta')=d', d(\beta'')=d''\}$$
of $\mathrm{Eff}(\CC^{\Sigma(1)},\G,\theta)\times \mathrm{Eff}(\CC^{\Sigma(1)},\G,\theta)$ is nonempty. We have a cartesian diagram
$$\begin{CD}S(\tau',\cP')\times S(\tau'',\cP'')@>>>\overline{M}_{0,k'+1|d(\beta')}\times \overline{M}_{0,1+k''|d(\beta'')}\\
@VVV @VVV\\
S(\tau,\cP)@>>>\overline{M}_{0,k|d(\beta)},
\end{CD}$$
where the horizontal maps are the inclusions and the vertical maps are obtained by gluing the last marking on the first factor to the first marking on the second factor.

Finally, given a monomial $\mu(\psi)=\psi_1^{\alpha_1}\dots\psi_k^{\alpha_k}$ we write  $\mu=\mu'\mu''$ with $\mu'(\psi)=\psi_1^{\alpha_1}\dots\psi_{k'}^{\alpha_{k'}}$
and $\mu''(\psi)=\psi_{k'+1}^{\alpha_{k'+1}}\dots\psi_{k}^{\alpha_{k}}$

\begin{Lemma}{\em (Splitting Lemma)}\label{Splitting Lemma} Let $k\geq 3$, $\beta\neq 0$, $\ke\in\{0+,\infty\}$,  and let $S(\tau,\cP)\subset \overline{M}_{0,k|d(\beta)}$  be a chain-type stratum. 

$(i)$ If $\cP$ is not compatible with $\beta$ then for every monomial $\mu(\psi)$ the intersection number
$$P_{\beta , \sigma }^{ \ke}(k,\tau, \cP, \mu )
:= \int _ {S(\tau, \cP)} P^\ke _{\beta , \sigma }(\hat{\psi} _j , D_J) \mu (\psi ) $$
vanishes.

$(ii)$ If  $\cP$ is compatible with $\beta$ then
\begin{align}\label{splitting} &P_{\beta , \sigma }^{ \ke}(k,\tau, \cP, \mu )=\\
\nonumber &=
\sum_{(\beta',\beta'')\in R(\cP,\beta)} P_{\beta' , \sigma }^{ \ke}(k'+1,\tau', \cP', \mu' )
P_{\beta'' , \sigma }^{ \ke}(k''+1,\tau'', \cP'', \mu'' ).
\end{align}
\end{Lemma}

\begin{proof} By an easy induction it suffices to assume $\ell=2$. We discuss the stable map case $\ke=\infty$; the quasimap case is similar (and easier). 

Let $\cP=(\cP',\cP'')$,
$\tau=(t',t'')$, so that $k'=t'+2$ and $k''=t''+2$. The corresponding stratum $S(\tau,\cP)$ is the image of the finite gluing morphism
$$j:\overline{M}_{0,k'+1|d'}\times \overline{M}_{0,1+k''|d''}\lra \overline{M}_{0,k|d(\beta)}.
$$

Let $\cF(k,\beta,\sigma)$ denote the set of graphs parametrizing $\T$-fixed loci in $\coprod_{A,\beta_0+\sum_a\beta_a=\beta}\overline{M}_{0,k\cup A}(X,\beta_0)$ 
supported at $p_\sigma$ under the contraction
$(b\circ c\circ i_\infty)^{\T}$, as described in \S\ref{localization}. For a graph $\Gamma\in\cF(k,\beta,\sigma)$ 
let $E_0$ be the set of edges incident to the distinguished vertex $v_0$. 
Each $e\in E_0$ has attached to it the class $\beta_e$ of the map from the corresponding rational tail $T_e$.
As explained in \eqref{ciT}, the contraction map on $F_\Gamma$ is essentially
$$g_\Gamma: \overline{M}_{g,k\cup A(v_0)\cup E_0}\lra \overline{M}_{g,k|A(v_0)\cup E_0}\lra \overline{M}_{g, k| d(\beta)}, $$
where the second map increases the multiplicity of each $e\in E_0$ by $d(\beta_e)+\sum_{a\in A(e)}d(\beta_a)$ and of each $a\in A(v_0)$ by $d(\beta_a)$. If the image of $g_\Gamma$ intersects the stratum, 
then there are set partitions
$E_0=E'_0\coprod E''_0$ and $A(v_0)=A'(v_0)\coprod A''(v_0)$ such that 
$$d'=d\left(\sum_{e\in E'_0}\left( \beta_e+\sum_{a\in A(e)}\beta_a\right)+\sum_{a\in A'(v_0)}\beta_a\right),$$
$$
d''=d\left(\sum_{e\in E''_0}\left( \beta_e+\sum_{a\in A(e)}\beta_a\right)+\sum_{a\in A''(v_0)}\beta_a\right).
$$
Part $(i)$ of the Lemma follows immediately from this, since the nonvanishing of the intersection number requires that the image of $g_\Gamma$ meets
$S(\tau,\cP)$ for at least one graph $\Gamma$.

Now let $\beta$ be compatible with $\cP$ and let $(\beta',\beta'')\in R(\cP,\beta)$. Applying the $\star$ operation
described in \S\ref{uniqueness} to $\Gamma'\in\cF(k'+1,\beta',\sigma)$ and 
$\Gamma''\in\cF(1+k'',\beta'',\sigma)$ we obtain a graph $\Gamma=\Gamma'\star\Gamma''\in\cF(k,\beta,\sigma)$, encoding the gluing 
map
$$h:\overline{M}_{0,k'+1 \cup E'_0\cup A'(v_0)}\times \overline{M}_{0,1+k'' \cup E''_0\cup A''(v_0)}\lra \overline{M}_{0,k \cup E_0\cup A(v_0)}.$$
Here $k=k'+k''$, $E_0=E'_0\coprod E''_0$, and $A(v_0)=A'(v_0)\coprod A''(v_0)$.
Further, we think of $[k'+1]$ as the set $\{1,2,\dots ,k'\}\cup\{\bullet\}$ and of $ [1+k'']$ as 
$\{\bullet\}\cup\{k'+1,k'+2,\dots k'+k''\}$, 
with the gluing done at the special markings $\bullet$.

The diagram 
$$\begin{CD}\overline{M}_{0,k' +1\cup E'_0\cup A'(v_0)}\times \overline{M}_{0, 1+k'' \cup E''_0\cup A''(v_0)}@>h>>\overline{M}_{0,k \cup E_0\cup A(v_0)}\\
@V{g_{\Gamma'}\times g_{\Gamma''}}VV @VV{g_\Gamma}V\\
\overline{M}_{0,k'+1 |d(\beta')}\times \overline{M}_{0,1+k'' |d(\beta'')}@>j>>\overline{M}_{0,k|d(\beta)},
\end{CD}$$ is cartesian, hence
\begin{equation}\label{cartesian} j^*(g_\Gamma)_*=(g_{\Gamma'}\times g_{\Gamma''})_*h^*.
\end{equation}
By \eqref{Mv}, the restriction of $P_{\beta , \sigma }^{ \infty}$ to $S(\tau,\cP)$ is computed by 
$$\sum_{\Gamma \in\cF(k,\beta,\sigma)}
j^*(g_\Gamma )_*\prod_{e\in E_0} \frac{\mathrm{cont}(T_e)}{\frac{w(e)}{\delta_e} - \psi _{e} }\prod_{a\in A(v_0)}\mathrm{cont}(a).$$
Since $\psi_e$ pulls-back under $h$ to $\psi_e\ot 1$ for $e\in E'_0$ and to $1\ot\psi_e$ for $e\in E''_0$, part $(ii)$ of the 
Lemma follows from \eqref{cartesian} by summing over
graphs.
\end{proof}

We are now in position to prove Lemma \ref{independence} and therefore complete the proof of Theorem \ref{ClassThm}.

Consider first the case when the class $\beta\neq 0$ is primitive, i.e., if $\beta=\beta'+\beta''$ with $\beta',\beta''\in\mathrm{Eff}(\CC^{\Sigma(1)},\G,\theta)$, then either $\beta'=0$, or $\beta''=0$. 
In this case, for every $k$ and $\mu$, we have
$P_{\beta , \sigma }^{ \ke}(k,\tau, \cP, \mu )=0$ whenever $\ell\geq 2$, 
by Lemma \ref{Splitting Lemma} $(a)$, while for $\ell=1$, the statement is given by the Uniqueness Lemma
(Lemma \ref{uniqueness lemma}).

Let now $\beta >0$ be arbitrary. Assume that if $0\neq\beta'\in \mathrm{Eff}(\CC^{\Sigma(1)},\G,\theta)$ is such that $\beta-\beta'$ is also nonzero and $\theta$-effective, then
$P_{\beta' , \sigma }^{ \ke}(k,\tau, \cP, \mu )$ is independent of $\ke$ for every $k$, every stratum $S(\tau,\cP)\subset\overline{M}_{0,k|d(\beta')}$, and 
every monomial $\mu(\psi)$
(or, equivalently by Lemma \ref{MOPlemma}, that when written in canonical form,
$P^{0+}_{\beta ',\sigma} = P^{\infty}_{ \beta ', \sigma}$ as abstract polynomials). 
Given a stratum $S(\tau, \cP)\subset\overline{M}_{0,k|d(\beta)}$, if $\ell=1$, we are done by  Lemma \ref{uniqueness lemma}. Otherwise, we split it at the first node and apply 
Lemma \ref{Splitting Lemma} $(b)$ to conclude by induction.

\subsection{Remarks on the proof and generalizations} 

\subsubsection{Other stability parameters for toric targets}\label{epsilon for toric} Even though we have restricted to the asymptotic stability condition $\ke=0+$, 
essentially the same argument works for general $\ke$. The required equality \eqref{I=J} between $J^\ke_{sm}(q,z)$ and the mirror transform $J(q, J_1^\ke , z)$ is provided
by Proposition \ref{summary}$(ii)$.
(In the case of the \cite{MOP} proof for Grassmannian targets, the extension to general $\ke$ is done in Toda's paper \cite{Toda}.)

\subsubsection{Other GIT presentations of a toric variety}\label{GIT presentations} For any GIT triple of the form $(\CC^{n+r},\G,\theta)$ with $\G\cong(\CC^*)^r$, satisfying our usual
assumptions that all semistable points are stable and $\G$ acts freely on the stable locus, the quotient $X=\CC^{n+r}/\!\!/_{\theta}\G$ is a nonsingular quasi-projective toric variety
of dimension $n$, see e.g., \cite{Dolgachev}, \S12.  Now
set $$\Sigma(1):=[n+r]=\{1,2,\dots, n+r\},$$
and write $\CC^{\Sigma(1)}$ for $\CC^{n+r}$. For a subset $\sigma\subset [n+r]$ of cardinality $n$, let 
$$\tilde{p}_\sigma:=(z_\rho) \in \CC^{\Sigma(1)}, z_\rho= \begin{cases} 1 & {\text{if}}\; \rho\notin \sigma\\ 0, & {\text{if}}\; \rho\in \sigma \end{cases} .$$
Put
$$\Sigma(n):=\{ \sigma\subset [n+r]\; |\:  |\sigma|=n\; {\text{and}}\; \tilde{p}_\sigma\; {\text{is $\theta$-stable}} \}
$$
and let $p_\sigma\in X$ denote the image of $\tilde{p}_\sigma$ under the quotient map. With the torus $\T=(\CC^*)^{\Sigma(1)}$ acting on $\CC^{\Sigma(1)}$ as before, we have a bijection
$\Sigma(n)\lra X^{\T}$, $\sigma\mapsto p_\sigma$. 

If $n\geq 1$, the (isolated) $1$-dimensional $\T$-orbits in $X$ at $p_\sigma$ are in bijection with the subsets $\tau$ of $\sigma$ of cardinality $n-1$ and are given by the 
vanishing of the corresponding
homogeneous coordinates $z_\rho$, $\rho\in\tau$. Such a $\T$-orbit is a $\PP^1$ if it contains another fixed point $p_{\sigma'}$, i.e., if $\tau=\sigma\cap\sigma'$ for some $\sigma'\in\Sigma(n)$.

With this expanded interpretation of the notations, the proof of Theorem \ref{ClassThm} is valid for all such general semi-positive GIT presentations. 

The small $I$-function associated to the triple $(\CC^{n+r},\G,\theta)$ is given by the same formula
\begin{equation}\label{toric small I 2}
I_{sm}(q,z)=\one +\sum_{\beta\neq 0} q^\beta \prod_{\rho=1}^{n+r}\frac{\prod_{j=-\infty}^0 (c_1^{\T}(\underline{L}_\rho)+jz)}{\prod_{j=-\infty}^{d_\rho}(c_1^{\T}(\underline{L}_\rho)+jz)},
\end{equation}
as it can be easily seen that the computation from \S7.2 of \cite{CK} works for all GIT presentations. 
Recall that the class $\beta$ runs over $\mathrm{Eff}(\CC^{n+r},\G,\theta)$ and $d_\rho=\beta(\xi_\rho)\in \ZZ$.

Note that it is possible now that the unstable locus contains components of codimension $1$ (of the form $\{z_\rho=0\}$, for some $\rho\in\Sigma(1)$),
and then $r=\mathrm{rk}(\chi (\G)) > \mathrm{rk}(\Pic(X))$. Consequently,
$\mathrm{Eff}(\CC^{n+r},\G,\theta)$ may contain strictly the Mori cone of $X$ if the GIT presentation is non-standard. 
In this case, the $I$-function \eqref{toric small I 2} will depend on additional ``ghost" Novikov parameters, and in particular will differ from the ``usual" small $I$-function coming from the
{\it standard} GIT presentation.

Semi-positivity of the triple $(\CC^{n+r},\G,\theta)$ still implies semi-positivity of the anti-canonical class of $X$, but
the converse is not necessarily true. For example, $\PP^2$ has a well-known GIT presentation $\CC^4/\!\!/_\theta (\CC^*)^2$ 
(obtained by variation of GIT from the standard presentation of the Hirzebruch surface $\mathbb{F}_1$) for which $(\CC^4,(\CC^*)^2,\theta )$ is {\it not} semi-positive.
We conclude the discussion by noting the following elementary fact.
\begin{Lemma}\label{I_0=1} If the triple $(\CC^{n+r},\G,\theta)$ is semi-positive, then the corresponding small $I$-function \eqref{toric small I 2} has
$I_0=1$.
\end{Lemma}
\begin{proof} Let $\beta\neq 0$. By semi-positivity, $\sum_\rho d_\rho\geq 0$.  The power of $1/z$ appearing in the $q^\beta$-term is equal to
$$\sum_\rho d_\rho + \#\{\rho \; |\; d_\rho< 0\}\geq \sum_\rho d_\rho \geq 0.$$
If we would have only equalities in the above chain, then $d_\rho=0$ for all $\rho$, which is impossible since $\beta\neq 0$. Hence the power of $1/z$ is strictly positive.
\end{proof}

\subsubsection{Other targets} It is clear that Theorem \ref{ClassThm} will be true, with the same proof, for all GIT targets $X=\WmodG$ corresponding to 
a semi-positive triple $(W,\G,\theta)$ and satisfying the following properties:

\begin{enumerate} 

\item The small $I$-function of $X$ has $I_0=1$.

\item The $\T$-action on $X$ has 
only isolated fixed points and isolated $1$-dimensional $\T$-orbits. (The isolated fixed points assumption insures that Proposition  \ref{summary}$(ii)$ will again provide the needed matching \eqref{I=J} of small $J$-functions.)

\item The push-forward under $c_\ke\circ\iota_\ke$ of the $\T$-vertex contributions on the qusimap moduli spaces $\QmapXe$
(i.e., the right-hand side of equation \eqref{eqvertexmatching}) 
can be written as a polynomial $P^\ke (\hat{\psi}, D_J)$ in $\hat{\psi}_j$ and $D_J$
(with coefficients in $H^*_{\T,\mathrm{loc}}(\Spec \CC)$) which is {\it independent on $g$ and $k$}.
\end{enumerate}

For the ``local Grassmannians" (i.e., the total space of the canonical bundle over
$Grass (r,n)$) the first two properties are immediate and the third is essentially shown in \cite{MOP}, \cite{Toda}. 
One can also easily check the third property when considering more general type A flag manifolds in place of Grassmannians. 
In particular, we have a proof of Theorem \ref{local Grass} as well. Note that explicit closed formulas for the small $I$-functions of these targets 
are easily obtained from the results in \cite{BCK1}, \cite{BCK2}, \cite{CKS}. For example, the $\T$-equivariant small $I$ for the local $Grass (r,n)$ is
\begin{equation}\label{local grass I}
\begin{split}
&I_{sm}=\one+\sum_{d>0}q^d   \left( \prod_{k=0}^{nd-1}(-n(\sum_{i=1}^rH_i)+\lambda_0 - kz)\times \right .\\
&\left . \sum_{d_1+\dots +d_r=d}\frac{(-1)^{(r-1)d}\prod_{1\leq i<j\leq r}(H_i-H_j+(d_i-d_j)z)}
{\prod_{1\leq i<j\leq r}(H_i-H_j)\prod_{i=1}^r\prod_{l=1}^{d_i}\prod_{j=1}^n
(H_i+\lambda_j+lz)}\right),
\end{split}
\end{equation}
where $H_1,\dots,H_r$ are the Chern roots of $S^\vee$, the dual of the
tautological subbundle (of rank $r$) $0\lra S\lra \cO^{\oplus n}$, and $\lambda_0,\lambda_1,\dots ,\lambda_n$ are the equivariant parameters for the torus $\T\cong\CC^*\times(\CC^*)^{n}$.
Here the first factor is the torus acting by scaling on the fibers of the canonical bundle, while the factor $(\CC^*)^n$ is the standard torus acting on the Grassmannian.

\end{document}